\documentclass[reqno]{amsart}

\usepackage[normalem]{ulem}

\usepackage{amsmath,amssymb,color}
\usepackage{amsfonts, amscd, epsfig, amsmath, amssymb,enumerate, bbm}
\usepackage{graphicx}
\usepackage{graphics}
\usepackage{color}
\usepackage{mathrsfs}
\usepackage{todonotes}
\usepackage{soul}
\usepackage[french,english]{babel}

 \usepackage[usenames,dvipsnames]{pstricks}
 \usepackage{pstricks-add}
 \usepackage{epsfig}
 \usepackage{pst-grad} 
 \usepackage{pst-plot} 
 \usepackage[space]{grffile} 
 \usepackage{etoolbox} 
\usepackage[margin=3cm]{geometry}
 \makeatletter 
 \patchcmd\Gread@eps{\@inputcheck#1 }{\@inputcheck"#1"\relax}{}{}
 \makeatother

\newtheorem{theorem}{Theorem}[section]
\newtheorem{lemma}[theorem]{Lemma}
\newtheorem{proposition}[theorem]{Proposition}
\newtheorem{corollary}[theorem]{Corollary}
\newtheorem{remark}[theorem]{Remark}
\newtheorem{definition}{Definition}[section]

\usepackage{tikz}
\usetikzlibrary{backgrounds}
\usetikzlibrary{patterns,fadings}
\usetikzlibrary{arrows,decorations.pathmorphing}
\usetikzlibrary{decorations}
\usetikzlibrary{calc}
\usetikzlibrary{shapes.misc}

\usepackage{setspace}

\definecolor{light-gray}{gray}{0.95}
\usepackage{float}
\usepackage[colorlinks=true,linkcolor=blue,citecolor=magenta]{hyperref}
\def\centerarc[#1](#2)(#3:#4:#5){\draw[#1] ($(#2)+({#5*cos(#3)},{#5*sin(#3)})$) arc (#3:#4:#5);}

\newcommand{\ccl}[1]{{\color{blue}#1}}

\newcommand{\vertiii}[1]{{\left\vert\kern-0.25ex\left\vert\kern-0.25ex\left\vert #1 
    \right\vert\kern-0.25ex\right\vert\kern-0.25ex\right\vert}}

\numberwithin{equation}{section}
\numberwithin{figure}{section}

\definecolor{MSBlue}{rgb}{.15,.0.35,.85}
\newcommand{\nota}[1]{{\color{MSBlue}#1}}

\newcommand{\mc}[1]{{\mathcal #1}}

\newcommand{\bb}[1]{{\mathbb #1}}

\newcommand{\<}{\langle}
\renewcommand{\>}{\rangle}

\renewcommand{\epsilon}{\varepsilon}

\newcommand{\R}{\mathbb R}

\newcommand{\Z}{\mathbb Z}
\newcommand{\N}{\mathbb N}
\renewcommand{\P}{\mathbb P}

\newcommand{\E}{\mathbb E}
\newcommand{\Q}{\mathbb Q}

\renewcommand{\bar}{\overline}

\newcommand{\abs}[1]{\;\left\vert\;#1 \;\right\vert\;}

\newcommand{\normm}[1]{{\left\vert\kern-0.1ex\left\vert\kern-0.1ex\left\vert\; #1 \; \right\vert\kern-0.1ex\right\vert\kern-0.1ex\right\vert}}    
\newcommand{\cro}[1]{\left[#1\right]}
\newcommand{\pa}[1]{\left(#1\right)}

\newcommand{\genzr}{\mathscr{L}^{\mbox{\tiny \textsc{zr}}}_N}
\newcommand{\genzrstar}{\mathscr{L}^{\mbox{\tiny \textsc{zr}},\star}_N}
\newcommand{\genzrh}{\widehat{\mathscr{L}}^{\mbox{\tiny \textsc{zr}}}_N}

\newcommand{\genex}{\mathscr{L}_N}

\newcommand{\sym}{\boldsymbol{\mathfrak{s}}}
\renewcommand{\leq}{\leqslant}
\renewcommand{\geq}{\geqslant}

\allowdisplaybreaks 

\title{Stationary fluctuations for the facilitated exclusion process}

\author{Cl\'ement Erignoux}
\address{Inria, Univ. Lille, CNRS, UMR 8524 - Laboratoire Paul Painlev\'e, F-59000 Lille}
\email{clement.erignoux@inria.fr}
\author{Linjie Zhao}
\address{School of Mathematics and Statistics, and Hubei Key Laboratory of Engineering Modeling and Scientific Computing, Huazhong University of Science and Technology, Wuhan 430074, China. }
\email{linjie\_zhao@hust.edu.cn}


\newenvironment{poliabstract}[1]
  {\renewcommand{\abstractname}{#1}\begin{abstract}}
  {\end{abstract}}
\begin{document}

\maketitle

\selectlanguage{english}
\begin{poliabstract}{Abstract} 
We derive the stationary fluctuations for the Facilitated Exclusion Process (FEP) in one dimension in the symmetric, weakly asymmetric and asymmetric cases.  Our proof relies on the mapping between the FEP and the zero-range process, and extends the strategy in \cite{erignoux2022mapping}, where hydrodynamic limits were derived  for the FEP, to its stationary fluctuations. Our results thus exploit works on the zero-range process's fluctuations \cite{gonccalves2010equilibrium,gonccalves2015stochastic}, but we also provide a direct proof in the symmetric case, for which we derive a sharp estimate on the equivalence of ensembles for the FEP's stationary states.
\end{poliabstract}



\section{Introduction}

The facilitated exclusion process (FEP) was first introduced in \cite{RossiPastorVespignani00} in the physics community, indicating the existence of a new universality class of nonequilibrium phase transitions in the presence of a conserving field.  In the FEP, particles can jump if and only if their target neighboring site is empty (exclusion rule) and another neighboring site is occupied (the facilitated rule). Particles with an occupied neighboring site are referred to as \emph{active}. As a result of the facilitated rule, the FEP is ultimately absorbed in some frozen configuration (all particles are isolated) if the particle density is below a critical value $\rho_c$ depending on the dimension, and can stay active for unbounded times otherwise.  Over the past years, much progress has been made in dimension one, where the critical value is $\rho_c:=1/2$. In a series of papers \cite{blondel2020hydrodynamic,blondel2021stefan,erignoux2022mapping}, the authors and their collaborators have shown that the macroscopic behaviors of the FEP are described by parabolic, resp. hyperbolic, Stefan problems when the process is symmetric, resp. asymmetric. The above results correspond to law of large numbers for the empirical density of the process and are called hydrodynamic limits in the literature \cite{klscaling}.

A natural question that arises once the hydrodynamic limit of the FEP is established is to consider fluctuations for the process, which play the role of the central limit theorem for interacting particle systems. Very recently, Barraquand \emph{et al.\;}\cite{barraquand2023weakly} investigated this issue when the process is weakly asymmetric and starts from step-like initial distributions. The specific form of the initial distribution allows them to map the process into the simple exclusion process on the half-line with particle creation (and no annihilation) at the origin. Different from their setting, we let the process start from its stationary measure and prove central limit theorems for the density fluctuation fields in the following three cases: symmetric, weakly asymmetric and asymmetric.   Roughly speaking, we show in this article that
\begin{enumerate}[(i)]
	\item in the symmetric case,  the density fluctuation fields converge to the solution of a stochastic heat equation;
	\item in the weakly asymmetric case with weak asymmetry of order $N^{-1}$, where $N$ is the scaling parameter, the density fluctuation fields converge to the solution of a stochastic Burgers-like equation;
	\item in the asymmetric case, the density fluctuation fields translates along characteristics of the corresponding hydrodynamic equation up to time $N^{\gamma}$ for $\gamma < 4/3$.
\end{enumerate}
In the proof of the symmetric case, we obtain a sharp estimate on the equivalence of ensembles (see Proposition \ref{pro:equiv}) for the FEP, improving the bound in \cite{blondel2020hydrodynamic}, which is of independent interest. The latter is essential to derive the so-called Boltzmann-Gibbs principle (see Proposition \ref{thm:bg}), which plays an important role when considering fluctuations from hydrodynamic limits.
 
\medskip
 
The mathematical novelty involved in proving the above results and the techniques in this article are the following:
\begin{itemize}
	\item In \cite{erignoux2022mapping},  the authors together with  M. Simon proved hydrodynamic limits for the FEP and FZRP, both in the symmetric and asymmetric cases by using a classical mapping between the FEP and the facilitated zero-range process (FZRP). In the present article, we extend the mapping technique to investigate the fluctuations of FEP. We remark that when the FEP starts from its stationary measure, the FEP can be mapped to the classical rate one zero-range process (ZRP), which is not degenerate. Despite this,  we still need to give sharp and uniform estimates on the positions of empty sites in order to express the density fluctuation field of the FEP as that of the ZRP, see Section \ref{sec:mappingproof} for details. We also need new non-trivial identities (see.\,\eqref{eq:mappingdistributions}) between the two process's invariant states. Once this is done, the derivation of macroscopic fluctuations for the FEP follow from previous results on the ZRP  \cite{gonccalves2010equilibrium,gonccalves2015stochastic}. Moreover, we believe the techniques used in this article  should also work for non-stationary fluctuations of FEP, but we leave this question as future work.
	\item Over the last ten years, much efforts have been made to understand the weak Kardar-Parisi-Zhang (KPZ) universality conjecture, which states that for weakly asymmetric systems with weak asymmetry of order $N^{-1}$ and with only one conservation law, the density fluctuation fields converge to the solution of the stochastic Burgers equation. Since the seminal work of \cite{gonccalves2014nonlinear}, this has been proven to be true for a large class of models \cite{gonccalves2015stochastic,diehl2017kardar}. The second contribution of the article is to provide another example of a degenerate system that validates the weak KPZ universality conjecture.
\end{itemize} 
  
\subsection {Related references.}  The FEP has been widely explored in recent years after being introduced in \cite{RossiPastorVespignani00}. In the physics literature, its critical behavior, namely its critical density and  critical exponents have been investigated in different dimensions in \cite{BasuMohanty09,oliveira,lubeck}. In \cite{GabelKrapivskyRedner10} the authors found the phenomenon of jump continuity at the leading edge of rarefaction waves, which is quite different from asymmetric simple exclusion.  In the mathematics literature, the stationary states of the facilitated exclusion, either in the continuous or discrete time setting, have been studied in \cite{Ayyer2020StationarySO,Chen2019,2019JSMTE,Goldstein2020,Goldstein2022}. Limit theorems have also been proved in  \cite{BaikBarraquandCorwinToufic16} for the position of the rightmost particle starting from step initial condition. We underline that the results mentioned above are concentrated on dimension one, and few rigorous results are known in higher dimensions.

The mapping between exclusion processes and the zero-range processes is well known, and has been put to use in varied contexts. It is used for instance in \cite{Kipnis86} in order to prove a central limit theorem for a tagged particle in the asymmetric simple exclusion process; in \cite{FunakiSasada}, the weakly asymmetric zero-range process with a
stochastic reservoir at the boundary (associated with the dynamics of two-dimensional Young diagrams) is mapped to the weakly asymmetric
simple exclusion process on the full line without any boundary condition, for which the hydrodynamic limit is known. In \cite{ferrari1988non}, the authors used the mapping when considering non-equilibrium fluctuations for the rate one zero-range process; more precisely, they proved a non-equilibrium version of the Boltzmann-Gibbs principle of the zero-range process, from which the non-equilibrium fluctuations follow directly.  Compared with \cite{ferrari1988non} we use the mapping in the opposite direction and map the two processes at the level of density fluctuation fields.

\subsection{Outline of the paper.} In Section \ref{sec:notation and results} we start by introducing the facilitated exclusion process, recalling its basic properties and stationary distributions.  Then, we introduce the object we are interested in, namely the fluctuations field, in Subsection \ref{subsec:flucfield} and the limiting processes (depending on the strength of the asymmetry) in Definitions \ref{def:solSHE} and \ref{def:solSBE}. We state our main results, characterizing the macroscopic  stationary fluctuations of the FEP, in Theorems \ref{thm:weak_asymFEP} and \ref{thm:asymFEP}. In Section \ref{sec:mapping} we describe the mapping between the FEP and a constant rate zero-range process, which is one of the main ingredients to derive our main result. Section \ref{sec:mappingproof} is devoted to the proof of  Theorems \ref{thm:weak_asymFEP} and \ref{thm:asymFEP}, the main idea being to express the density fluctuation field of the FEP as that of the ZRP. In Section \ref{sec:proofEqui} we give a alternative proof in the symmetric case, where a sharp estimate on the equivalence of ensembles is obtained, see Proposition \ref{pro:equiv}. Some new and useful properties of the stationary measures of the FEP are given in Appendix \ref{sec:stationaryDist}. 

\subsection{General Notations.} To ease reading, we indicate \nota{in color} the new notations that are introduced inside of paragraphs. We are interested in the stationary fluctuations (under the stationary distribution $\nota{\pi_\rho}$ at density $\rho\in[0,1]$ fixed. Throughout, given a random variable $X$, we will denote by $\nota{\overline{X}=X-\E_{\pi_\rho}(X)}$ the corresponding centered variable. $\nota{\N:=\{0,1,2,\dots,\}}$ is the set of non-negative integers. 

Since we are working on two different processes, it is convenient to introduce distinct notations for these two processes and we summarize them below.
\begin{center}
	\begin{tabular}{|l|c|c|}
	\hline
		& \qquad  FEP \qquad  \quad         & \qquad  ZRP \qquad \quad         \\ 
		 \hline
		Microscopic space variable & $x$ & $y $ \\
		Configuration       & $\eta_x$       & $\omega_y$     \\ 
		Macroscopic space variable & $u$ & $v:=(1-\rho)u $ \\
		macroscopic density & $\rho (u)$ & $\alpha(v)$ \\
		Stationary distributions & $\pi_\rho$   & $\mu_\rho$   \\
		Distribution of the stationary trajectory          & $\mathbb{P}_ {\rho}^N$ & $\mathbf{P}_{\nu}^N$ \\ 
		Expectation w.r.t. the trajectory's distribution          & $\mathbb{E}^N_{\rho}$ & $\mathbf{E}_{\nu}^N$ \\ 
		\hline
	\end{tabular}
\end{center}
\smallskip
Note that the static distributions for the zero-range process is parametrized by the exclusion process's density $\rho$ rather than the zero-range density $\alpha$, but they could straightforwardly be expressed as distributions  of the $\alpha$ through equation \eqref{eq:defalpha} below. In the case of the trajectory's distribution, we parametrize by the initial state $\nu$ of the system.

\section{Notation and results}\label{sec:notation and results}

\subsection{Facilitated exclusion process}
 Let  $\nota{N\in \N}$ be the scaling parameter for our process.   The \emph{facilitated exclusion process on $\Z$} is a Markov process on the set of configurations $\nota{ \Sigma:=\{0,1\}^{\Z}}$.  A \emph{configuration} $\nota{\eta\in \Sigma}$ is sequences of $0$'s and $1$'s indexed by $\Z$, namely  $\eta_x=1$ if and only if  site $x \in \Z$ is occupied by a particle.
The infinitesimal generator ruling the evolution in time of this Markov process is given by $\genex$, which acts on local functions $f:\Sigma \to \R$ as 
\begin{equation}
\label{eq:DefLN}
\genex f(\eta):=\sum_{x\in \Z } \big( p_N c_{x,x+1}(\eta) + q_N c_{x+1,x} (\eta) \big) \big(f(\eta^{x,x+1})-f(\eta)\big) ,
\end{equation}
where 
\begin{equation}
\label{eq:jumprates}
p_N = \sym N^2 + N^{\gamma} \qquad\mbox{and}\qquad  q_N =  \sym N^2,
\end{equation}
where $\sym\in\{0,1\}$, $\gamma\in \R\cup\{-\infty\}$, and $\eta^{x,x'}$ denotes the configuration obtained from $\eta$ by swapping the values at sites $x$ and $x'$,
\[
\eta_z^{x,x'}=
\begin{cases}
\eta_{x'} & \mbox{ if } z=x,\\
\eta_x & \mbox{ if } z=x',\\
\eta_z & \mbox{ otherwise.}
\end{cases}
\]
The jump rates $c_{x,x'}(\eta)$ in \eqref{eq:DefLN} encode two dynamical constraints: 
\begin{enumerate}[(i)]
\item the \emph{exclusion rule},  which imposes no more than one particle at each site,
\item the \emph{facilitated rule}, a kinetic constraint which asks for a neighboring occupied site  in order for a particle to jump to the other neighboring empty site.
\end{enumerate}
More precisely, we define
\begin{equation} 
\label{eq:rate} 
c_{x,x+1}(\eta)=\eta_{x-1}\eta_x(1-\eta_{x+1}), \quad c_{x+1,x} (\eta) = (1-\eta_{x})\eta_{x+1}\eta_{x+2}.
\end{equation}

\medskip

In this article, we consider the following two cases:
\begin{itemize}
	\item the \emph{weakly asymmetric} case where $\sym=1$ and  $\gamma \leq 3/2$;
	\item the \emph{asymmetric case}, where $\sym=0$ and  $\gamma < 4/3$.
\end{itemize}
Note that letting $\sym=1$ and $\gamma=-\infty$, we obtain the \emph{symmetric} case, which is therefore considered as a special case of the weakly asymmetric one.

\medskip

Let us now recall some results from previous work \cite{blondel2020hydrodynamic,blondel2021stefan}: because of the kinetic constraint (ii), the facilitated exclusion process displays a \emph{phase transition} at the critical density $\rho_c=\frac12$. Indeed, pairs of neighboring empty sites cannot be created by the
dynamics, because to do so would require an isolated particle jumping out. Therefore, if at initial time the density of particles $\rho$ is bigger than $\rho_c$ (at least half of the sites are occupied), then particles will perform random jumps
in the microscopic system until there are no longer two neighboring empty sites. Similarly, if
initially $\rho<\rho_c$  (at least half of the sites are empty), particles will perform random jumps until all particles can no longer move because they are surrounded by empty sites. The particle configurations can therefore be divided into three categories: 
\begin{itemize}
\item the \emph{ergodic} configurations, where all empty sites are isolated, namely:

 $\eta$ is \emph{ergodic} if, for any $x\in \Z$, $\eta_x+\eta_{x+1}\geqslant 1$;
 We denote, for any connected set $B\subset \Z$, by 
 \begin{equation}
 \label{eq:DefEB}
 \mathcal{E}_B:=\left\{\eta\in \{0,1\}^B,\; \eta_x+\eta_{x+1}\geq 1\; \forall (x,x+1)\in B\right\}
 \end{equation} 
 the set of ergodic configurations on $B$.

\item the \emph{frozen} configurations, where all particles are isolated, namely: 

$\eta$ is \emph{frozen} if, for any $x\in\Z$, $\eta_x+\eta_{x+1}\leqslant 1$; Similarly, for a box $B\subset \Z$, we denote by 
 \[\mathcal{F}_B:=\left\{\eta\in \{0,1\}^B,\; \eta_x+\eta_{x+1}\leq  1\; \forall (x,x+1)\in B\right\}\]
 the set of ergodic configurations on $B$.

\item the \emph{transient} configurations, which are the remaining ones, those which are neither ergodic, nor frozen. They are called transient in \cite{blondel2020hydrodynamic,blondel2021stefan} because starting from a transient configuration, restricting the dynamics to a finite box $B$, the microscopic process will locally evolve towards either the ergodic or frozen components on $B$ in an a.s. finite number of jumps. 
\end{itemize} 
\noindent Note in particular that  the alternated configurations, where each particle is surrounded by empty sites and vice-versa, are critical, since their density is exactly $\frac12$, and it is convenient for them to be considered both frozen \emph{and} ergodic:  they are indeed frozen (no particle is allowed jump in them), and because they have probability non-zero under the grand canonical distribution of the process, whose support we refer to as the \emph{ergodic component}, it is natural to see them as ergodic as well.

\medskip

As a consequence, the  infinite volume invariant measures of the facilitated process are not independent products of homogeneous Bernoulli
measures (as in the standard Simple Exclusion Process for instance). Instead, they are given for $\rho>\frac12$, $\ell\geq 1$, by the measures
\begin{equation}
\label{pirho2}
\pi_\rho\Big( \eta_{|\Lambda_\ell}={\sigma}\Big)={\bf{1}}_{\{\sigma \in \mathcal{E}_\ell\}}
(1-\rho)\left(\frac{1-\rho}{\rho}\right)^{\ell-1-p}\left(\frac{2\rho-1}{\rho}\right)^{2p-\ell+1-{\sigma(1)-\sigma(\ell)}},
\end{equation}
where $\nota{\Lambda_\ell := \{1,\dots,\ell\}}$ for $\ell \geq 1$,  $\nota{\mathcal{E}_\ell:=\mathcal{E}_{\Lambda_\ell}}$ denotes the set of configurations which are ergodic on $\Lambda_\ell$, and $\nota{p = p(\ell,\sigma) := \sum_{x \in \Lambda_\ell} \sigma_x}$ is the number of particles in $\Lambda_\ell$. On can think of these measures as product Bernoulli distributions conditioned to being in the ergodic component. Furthermore,  they are translation invariant, and one can easily check that under $\pi_\rho$, we have $\E_{\pi_\rho}(\eta_0)=\rho$. The formula above, however, is not always very convenient to compute local function's expectations. For this reason, we describe below in \eqref{eq:mappingdistributions} an alternative construction for $\pi_\rho.$

\medskip

Consider as initial distribution the infinite volume stationary measure $\pi_\rho$, and the process $\eta(t)$ started from $\pi_\rho$ and with generator $\genex$ on $\mathcal{E}_{\Z}$. We denote by $\nota{\P_{\rho}^N}$ its distribution, and by $\nota{\E_{\rho}^N}$ the corresponding expectation.

\subsection{Stationary fluctuations field}\label{subsec:flucfield} Denote by $ \nota{\mathcal{S}:=\mathcal{S}(\R)}$ the Schwartz space of functions with fast decaying derivatives on $\R$, and by $\nota{\mathcal{S}'}$ its dual space, namely the space of tempered distributions. Further denote by  $\nota{D([0,T],\mathcal{S}')}$ the set of c\`adl\`ag trajectories on $\mathcal{S}'$. We now define the density fluctuation field  trajectory $(\mc{Y}^N_t)_{t\geq 0}\in D([0,T],\mathcal{S}')$ of the process $\eta$,  acting on smooth compactly supported functions $G\in \mathcal{S}$  as
\begin{equation}
\label{eq:YNt}
\mc{Y}^N_t (G) = \frac{1}{\sqrt{N}} \sum_{x \in \Z} \bar{\eta}_x (t) G \big(\tfrac{x}{N} - t v N^{\gamma -1}\big),
\end{equation}
where $\nota{\bar{\eta}_x (t) = \eta_x (t) - \rho}$ and $\nota{v :=v(\rho)}$ is the average macroscopic speed of asymmetric particles in the stationary state, namely  
	\begin{equation}
	\label{eq:defv}
	v(\rho) :=  \frac{d}{d \rho}\E_{\pi_\rho}(\eta_{x-1}\eta_x(1-\eta_{x+1}))= \frac{d}{d \rho}\frac{(1-\rho)(2\rho-1)}{\rho} = \frac{1-2\rho^2}{\rho^2}.
	\end{equation}
Note that because of the asymmetry, mass moves rightwards at a velocity of order $O( N^{\gamma -1})$, which is the reason for shifting in time the test function by $vtN^{\gamma-1}$. 
In what follows, we will use the notation 
\begin{equation}
\label{eq:defvN}
v_N=v N^{\gamma-1},
\end{equation}
for the macroscopic velocity of particles in the system in its stationary state. 

\medskip

We denote by $\nota{\Q_{\rho}^N:=\P_{\rho}^N\circ (\mc{Y}^N)^{-1}}$ the pushforward of $\P_{\rho}^N$ through the mapping 
\begin{equation*}
\mc{Y}^N:\{\eta(t),\; t\geq 0\}\mapsto \{ \mc{Y}^N_t,\;t\geq 0\},
\end{equation*}
that is, $\Q_{\rho}^N$ is the distribution of the stationary fluctuation field $\mc{Y}^N_t$ when $\eta$ is stationary state $\P_\rho^N$. In the asymmetric and weakly asymmetric cases  mentioned above, we are now ready to define the limiting processes and  introduce our main results.

\subsection{Diffusion coefficient and compressibility}
We start by introducing the relevant macroscopic quantities to appear in the macroscopic fluctuation field of the FEP. The hydrodynamic behavior of the symmetric FEP (see \cite{blondel2021stefan}) is characterized by the Stefan problem 
\[\partial_t \rho=\partial_u^2 \big\{ a(\rho) {\bf 1}_{\{\rho\geq 1/2\}}\big\}=\partial_u \big\{ D(\rho) {\bf 1}_{\{\rho\geq 1/2\}} \partial_u \rho\big\},\]
where $a(\rho)$ average density of active particles under $\pi_\rho$ and $D(\rho)$ the diffusion coefficient,
\begin{equation}
\label{eq:arho}
a(\rho):=\pi_\rho(\eta_x=1\mid \eta_{x-1}=1)=\frac{2\rho-1}{\rho}\qquad  \mbox{ and }\qquad D(\rho)=a'(\rho)=\frac{1}{\rho^2}.
\end{equation}
We further define  the conductivity 
\begin{equation}
\label{eq:Defsigmarho}
\sigma(\rho):=\pi_\rho(c_{0,1})=\frac{(2\rho-1)(1-\rho)}{\rho}.
\end{equation}
and the compressibility as
\begin{equation}
\label{eq:Defchirho}
\chi(\rho):=\sum_{x\in \Z}Cov_{\pi_\rho}(\eta_0,\eta_x)=\rho(1-\rho)(2\rho-1).
\end{equation}
Identities \eqref{eq:arho} and \eqref{eq:Defsigmarho} are direct applications of the explicit formula \eqref{pirho2}. We prove in Appendix \ref{sec:comp} this explicit formula \eqref{eq:Defchirho}  for the compressibility. Note that, as expected, the Einstein (fluctuation-dissipation) relation $\sigma=D\chi$ holds.

\subsection{Stationary fluctuations} In order to introduce our main results, we start by choosing $\sym=1$ and $\gamma \leq 3/2$ in \eqref{eq:jumprates} to consider the weakly asymmetric case.  We denote by $K_t(u,v)$ the one-dimensional heat kernel 
\begin{equation}
\label{eq:DefK}
K_t(u,v):=\frac{1}{\sqrt{4\pi t}}e^{\frac{(u-v)^2}{4t}},
\end{equation}
we now introduce the notion of solution to the stochastic heat equation for the case $\gamma<3/2$.
\begin{definition}[Solution to the stochastic heat equation]
\label{def:solSHE}
We say that $Y:=(Y_t)_{t\geq 0}$ taking values a.s. in $C([0,T], \mathcal{S}')$ is a stationary solution to the stochastic heat equation
\begin{equation}
\label{eq:DefSHE}
\partial_t Y_t = D(\rho)  \partial_u ^2 Y_t  + \sqrt{2\sigma(\rho)}\partial_u \dot{ \mathcal{W}}_t
\end{equation}
if it is a stationary generalized Ornstein-Uhlenbeck process on $\mathcal{S}'$ with mean $0$ and covariance given by
\[\E(Y_t(G)Y_s(H))=\chi(\rho) \int_{\R \times \R} G(u) K_{(t-s)D(\rho)}(u,v) H(v) du dv=\chi(\rho) \langle T_{t-s}G,H\rangle,\]
for $G,H\in \mathcal{S}$ and $s\leq t$, where $T_t$ is the semi-group associated with the self-adjoint operator $D(\rho) \partial_u^2 $ and $\langle\cdot ,\cdot\rangle$ is the $L^2$ inner product on $\R$.
\end{definition}

We will use the following result, which states that solutions to the stochastic heat equation are solution to a martingale problem.
\begin{proposition}[Characterization of solutions to \eqref{eq:DefSHE}]
\label{prop:carsolSHE}
Fix a stochastic process $Y$ taking values a.s. in $C([0,T], \mathcal{S}')$, and assume that for any $G \in \mathcal{S}$,
\[M_t (G) := Y_t (G) - Y_0 (G) -  D(\rho) \int_0^t Y_s (\partial_u^2 G) ds\]
and 
\[N_t(G) :=\big[M_t (G)\big]^2 - 2 t   \sigma(\rho) \|\partial_u G\|_{L^2 (\R)}^2\]
	are both integrable martingales w.r.t. $Y$'s natural filtration, and that for any $t\geq 0$, 
	\begin{equation}
	\label{eq:fixedtimecovariance}
	\E(Y_t(G)Y_t(H))=\chi(\rho) \langle G,H\rangle.
	\end{equation}
	Then, $Y$ is a stationary solution to the stochastic heat equation in the sense of Definition \ref{def:solSHE}. 
\end{proposition}

In the case $\gamma = 3/2$, a non-linear contribution coming from the asymmetric jumps appears in the fluctuation regime. To define the proper limiting equation, we need to introduce further definitions. First, define 
\[\iota (u) = (1/2) \mathbf{1}_{[-1,1]} (u),\]
and let $\varphi\in \mathcal{S}$ be a mollifier on $\R$, namely a non-negative compactly supported function such that $\int_\R\varphi(u) du =1 $. We then define \[\varphi_\varepsilon(u)=\varepsilon^{-1} \varphi(u/\varepsilon) \qquad \mbox{ and } \qquad \iota_\varepsilon (u) = \varepsilon^{-1} \iota(u/\varepsilon).\]
We now approximate the dirac measure in $u$ by the smooth convolution $\delta_{u,\varepsilon}=\varphi_{\varepsilon^3}*\iota_\varepsilon(\cdot-u)\in \mathcal{S}$, straightforward computations show that 
\[\|\delta_{0,\varepsilon}\|_{L^2 (\R)}^2 \leq \varepsilon^{-1} \quad \text{and} \quad \lim_{\varepsilon \rightarrow 0} \varepsilon^{-1/2} \|\delta_{0,\varepsilon} - \iota_\varepsilon\|_{L^2 (\R)} = 0.\]
We now introduce a necessary $L^2$ condition in the case where $\gamma=3/2$.
\begin{definition}[$L^2$ Energy condition]
\label{def:L2EC}
Fix a process $Y$ taking values a.s. in $C([0,T], \mathcal{S}')$, and for $0 \leq s \leq t \leq T$, we define $A_{s,t}^\varepsilon\in \mathcal{S}'$ by
\[A_{s,t}^\varepsilon (G) = \int_s^t \int_\R  Y_{s'} (\delta_{u,\varepsilon})^2 \partial_u G (u) du\,ds'.\]
We say that $Y$ \emph{satisfies the $L^2$ energy condition} if for any $G \in \mathcal{S} (\R)$, $A_{s,t}^\varepsilon (G) $ is Cauchy in $L^2$ as $\varepsilon \rightarrow 0$, and the limit in $L^2$ does not depend on the mollifier $\varphi$.
We then denote by $A_{s,t}=A_{s,t}[Y]$ the limit of $A_{s,t}^\varepsilon$.
\end{definition}
Note that we do not know a priori whether $A_{0,t}$ is time-differentiable, therefore we do not give an analogous definition to \ref{def:solSHE} for solutions to the stochastic Burgers equation. Instead, we directly use the characterization of solutions in terms of martingales, analogous to Proposition \ref{prop:carsolSHE}.

\begin{definition}[Solution to the stochastic Burgers equation]
\label{def:solSBE}
We say that a stochastic process $Y$ taking values a.s. in $C([0,T], \mathcal{S}')$ is a stationary solution to the stochastic Burgers equation 
\begin{equation}\label{sbe_zrp}
	\partial_t Y_t = D(\rho)  \partial_u^2 Y_t  + \frac{1 }{2} D'(\rho)\partial_u Y_t^2+ \sqrt{2 \sigma(\rho)} \partial_u \dot{ \mathcal{W}}_t.
\end{equation}
if
\begin{enumerate}[i)]
\item For any $t>0$, and $G, H\in \mathcal{S}$,
	\begin{equation}
	\label{eq:fixedtimecovariance}
	\E(Y_t(G)Y_t(H))=\chi(\rho) \langle G,H\rangle.
	\end{equation}
\item $Y$ satisfies the $L^2$ energy condition (cf. Definition \ref{def:L2EC}), so that for any $t\geq 0$, the tempered distribution $A_{0,t}[Y]\in \mathcal{S}'$ is well-defined.
\item For any $G \in \mathcal{S}$,
\[M_t (G):= Y_t (G) - Y_0 (G) -  D(\rho) \int_0^t Y_s (\partial_u^2 G) ds +\frac{1}{2}D'(\rho)A_{0,t}(G)\]
\[N_t(G):=\big[M_t (G)\big]^2 - 2 t   \sigma(\rho) \|\partial_u G\|_{L^2 (\R)}^2\]
	are both integrable martingales w.r.t. $Y$'s natural filtration.
\end{enumerate}
\end{definition}

We are now in a position to state our main result, which derives the fluctuation field in the weakly asymmetric case.
\begin{theorem}[Weakly asymmetric case]
\label{thm:weak_asymFEP}
For $\sym=1$, the FEP's  fluctuation field $\{\mc{Y}^N_t, \;0\leq t\leq T\}$ introduced in \eqref{eq:YNt}  converges in the uniform topology on $D\big([0,T], \mathcal{S}'(\R)\big)$, as $N\to\infty$ to a process $\{Y_t, \; 0\leq t\leq T\} $, which is 
\begin{enumerate}[i)]
\item solution to the stochastic heat equation in the sense of Proposition \ref{prop:carsolSHE} for $\gamma < 3/2$,
\item solution to the stochastic Burgers equation in the sense of Definition  \ref{def:solSBE} for $\gamma=3/2$.
\end{enumerate}
This convergence is to  be understood as a weak uniform convergence of the distribution of $\mathcal{Y}_N $ to that of $Y$.
\end{theorem}
To give a more explicit description of this convergence, by Skorokhod's representation theorem, in both cases one can build the limiting process $\{Y_t, \; 0\leq t\leq T\} $ on the same probability space $(\Omega, \mathscr{F}, \bb{P})$ as the facilitated exclusion process $\{\eta_t, 0\leq t\leq T\}$, and have that for any $G\in \mathcal{S}$
\[\limsup_{N\to\infty}\sup_{0\leq t\leq T}|\mc{Y}^N_t(G)-Y_t(G)|=0 \qquad \bb{P}\mbox{-a.s.}.\]
Note that this result does not seem to explicitly depend on the value of $\gamma$, except through the condition $\gamma =3/2$ or $\gamma <3/2$, which may seem strange since the particle's motion does. This is natural, however, because we are looking at the scaling limit of the moving field $\mc{Y}^N_t$, which translates at a $\gamma$-dependent speed.

\bigskip

We now consider the (totally) asymmetric FEP, and let $\sym=0$ and $\gamma < 4/3$ in \eqref{eq:jumprates}, in which case we have the following result.

\begin{theorem}[Totally asymmetric case]
\label{thm:asymFEP}
For  $\sym=0$ and $\gamma < 4/3$, the FEP's  fluctuation field $\{\mc{Y}^N_t, \;0\leq t\leq T\}$ introduced in \eqref{eq:YNt}  converges in the weak topology on $D\big([0,T], \mathcal{S}' \big)$, as $N\to\infty$ to a stationary Gaussian process $Y$ with covariance
\[\E(Y_t(G)Y_s(H))=\chi(\rho) \langle G,H\rangle.\]
In other words, the limiting process $Y_t=\sqrt{2\sigma(u)}\partial_u \mathcal{W}_t$ is a time-integrated space-time white noise, or, equivalently, the space derivative of  a $(2,1)$-Brownian sheet.
\end{theorem}
Once again, the limiting process does not depend on the asymmetry exponent $\gamma$, because the field $\mc{Y}^N_t$ we are looking at is taken in a moving frame that does depend on $\gamma$.

\bigskip
The proof of Theorems \ref{thm:weak_asymFEP} and \ref{thm:asymFEP} is the purpose of Section \ref{sec:mappingproof}. In order to prove these two results, we will rely on a classical mapping to an attractive zero-range process, already exploited to derive the FEP's hydrodynamic limit in both symmetric and weakly asymmetric cases, and described in details in Section \ref{sec:mapping}. The derivation of the stationary fluctuations field's scaling limit then follows from a regularity estimate on the mapping in the stationary state (cf. Proposition \ref{prop:error} below), together with the derivation of stationary fluctuation already obtained in the weakly asymmetric case in \cite{gonccalves2015stochastic} and in the totally asymmetric case in  \cite{gonccalves2010equilibrium}.  

\bigskip

In Section \ref{sec:proofEqui}, we then offer a direct, alternative proof of Theorem \ref{thm:weak_asymFEP} in the symmetric case $\sym=1$, $\gamma=-\infty$. The latter relies on a sharp bound (of order $\ell^{-1}\log^2 \ell $) for the equivalence of ensembles (cf. Proposition \ref{pro:equiv} below) which  is interesting in its own right since it significantly improves on the one obtained so far in \cite[Proposition 6.9, p. 697]{blondel2020hydrodynamic} (of order $\ell^{-1/4}$).

\section{Mapping to the constant rate zero-range process}
\label{sec:mapping}
We start by describing a classical mapping between the facilitated exclusion process and a zero-range process. In the ergodic component $\mathcal{E}_\Z$ (cf. \eqref{eq:DefEB}), we have no ergodicity issues, and this mapping can be built straightforwardly to the unconstrained, constant-rate zero-range process.

\subsection{Static mapping}
\label{sec:staticmapping}
Recall the definition \eqref{eq:DefEB} of the ergodic component $\mathcal{E}_\Z$, and fix an ergodic configuration $\eta\in \mathcal{E}_\Z$. Let $\nota{X_0=X_0(\eta)\leq 0}$ be the position, of the first empty site in $\eta$ to the left of (or at) the origin, meaning that $\eta_{X_0}=0$ and  $\eta_x=1$ for  $X_0< x\leq 0$. For any $y>0$ (resp.~$y<0$), we define $\nota{X_y}$ as the position of the $y$-th empty site to the right (resp. to the left) of  $X_0$. We then define 
\begin{equation}
\label{eq:omegaXy}
\omega_y := X_{y+1} - X_y - 2.
\end{equation}
Recall  $\nota{\N=\{0,1,2,\dots\}}$, and note that $\omega_y \in \N$, so that $\omega\in \N^\Z$ defines a zero-range configuration, because $\eta$ was assumed to be ergodic. Of course, the mapping $\eta\mapsto \omega$ is not $1$-to-$1$, since knowledge of the position of the initial empty site $X_0$ is necessary to revert the construction.

\medskip
\begin{figure}
\includegraphics[width=13cm]{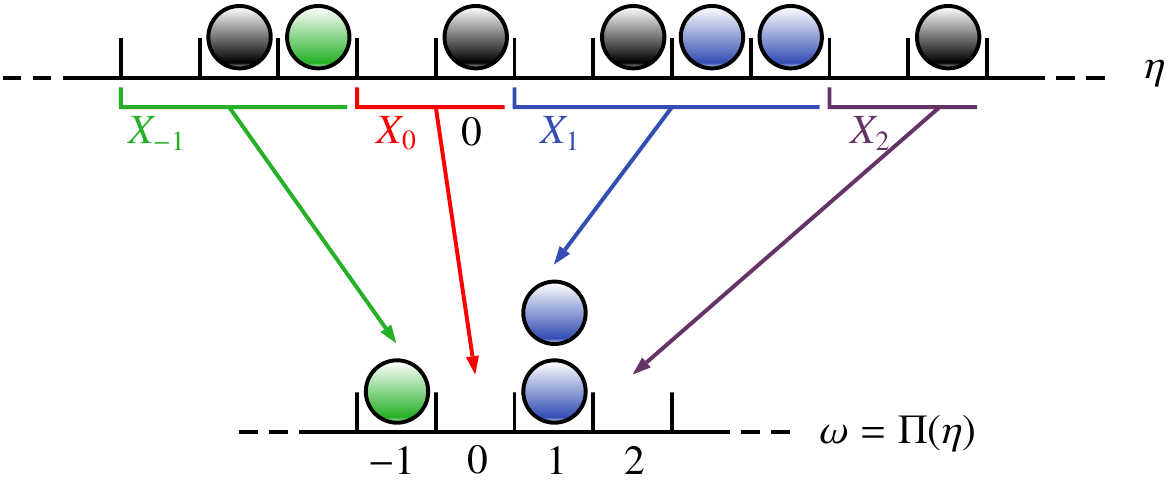}
\label{Fig:Mapping}
\caption{Representation of the mapping $\Pi$ between a FEP ergodic configurations $\eta$ and zero-range configuration $\omega$}
\end{figure}

Instead, denote by $\nota{\Pi:\eta\mapsto (\omega, X_0)}$, the mapping represented in Figure \ref{Fig:Mapping} which, with an ergodic configuration $\eta$, associates the zero-range configuration $\omega\in \N^\Z$ built through \eqref{eq:omegaXy} and the position of the first empty site left of the origin. It is  straightforward to show that the mapping $\Pi$ is a bijection between the FEP's ergodic configuration $\mathcal{E}_\Z$, and the set 
\begin{equation}
\label{eq:EZhat}
\widehat{\mathcal{E}}_\Z:=\{(\omega, X_0)\in \N^\Z\times \Z \; \vert \;-\omega_0-1\leq  X_0\leq 0\}.
\end{equation}
The reverse mapping $\Pi^{-1}$ is easily defined; choose a zero-range configuration $\omega\in \N^\Z$ and  an integer $-\omega_0-1\leq X_0\leq 0$. The associated ergodic configuration $\eta$ is then built by placing an empty site in $\eta$ at site $X_0$, and then choosing the consecutive positions $X_y$ of empty sites in $\eta$ according to $\omega$ and \eqref{eq:omegaXy}. We then have as wanted $\eta=\Pi^{-1}(\omega,X_0)$. 

\subsection{Dynamic mapping}
\label{sec:Dynmap}
We now consider a dynamical version of the mapping above : to do so, consider a trajectory $\{\eta(t), \; t\geq 0\}$ of the FEP starting from an ergodic configuration $\eta(0)\in \mathcal{E}_\Z$. Let $\nota{X_0(0)}$ be the position of the first empty site in $\eta(0)$ to the left of (or at) the origin at the initial time. Then, as before we define for any $y>0$ (resp.~$y<0$), $\ccl{X_y(0)}$ is the position of the $y$-th empty site to the right (resp. to the left) of $X_0(0)$.  We keep track of each trajectory $\ccl{\{X_y(t),\; t\geq 0\}}$ of the $y$-th empty site for $y\in \Z$ in $\{\eta(t), \; t\geq 0\}$.
Since the jumps are nearest neighbor, the orders of the empty sites are preserved along the evolution of the process, \emph{i.e.}~for any $t \geq 0$,
\[\ldots< X_{-1} (t) < X_0 (t) < X_1 (t) < \ldots.\] 
We then define as in the static case
\begin{equation}
\label{eq:omegaXyt}
\omega_y (t) := X_{y+1} (t) - X_y (t) - 2\quad  \mbox{ for }y\in \Z.
\end{equation}
In order not to confuse with the static mapping, we denote $\nota{\Pi^\star[\eta]}$ this dynamic mapping between trajectories, meaning that $\Pi^\star[\eta](t)=\omega(t)$. Note that we do not have in general that $\Pi^\star[\eta](t)=\Pi(\eta(t))$, unless the tagged empty site that was at time $0$ the first left of the origin is still the first left of the origin at time $t$.

\medskip

It is straightforward to show that if $\eta$ is a FEP, $\{\omega (t), \;t\geq 0\}=\Pi^\star[\eta]$ evolves as the classical zero-range process with generator $\genzr$,  which acts on local functions $f: \N^\Z \rightarrow \R$ as 
\begin{equation}
\label{eq:Lzr}
\genzr f (\omega) = \sum_{y \in \Z}  \mathbf{1}_{\{\omega_y \geq 1\}}\big\{ p_N  \big(f(\omega^{y,y+1}) - f(\omega) \big) +  q_N  \big(f(\omega^{y,y-1}) - f(\omega) \big) \big\}.\end{equation}
Since we use different letters for exclusion ($\eta$) and zero-range ($\omega$) configurations, without confusion, we also denote $\omega^{y,y\pm 1}$ the zero-range configuration obtained from $\omega$ after a particle jumps from $y$ to $y\pm 1$,
\[
\omega_{y'}^{y,y\pm 1}=
\begin{cases}
\omega_y-1 & \mbox{ if } y'=y,\\
\omega_y+1 & \mbox{ if } y'=y\pm 1,\\
\omega_{y'} & \mbox{ otherwise.}
\end{cases}
\]
Given a distribution $\nu$ on $\N^\Z$, we denote by   $\nota{{\bf P}^N_{\nu}}$  the distribution of the zero-range process started from the initial state $\omega(0)\sim \nu$ and driven by the generator $\genzr$ and by $\nota{{\bf E}^N_{\nu}}$ the corresponding expectation.

\subsection{Stationary states}
\label{sec:SSZR} 
We now consider the stationary states for the zero-range process. We define the (grand-canonical) equilibrium distributions for the generator $\genzr$ as the product distributions with marginals in $\omega_y$ given by geometric distributions with parameter the active density $a(\rho)$ (cf.  \eqref{eq:arho}), namely
\begin{equation}
\label{eq:Defmurho}
\mu_\rho (\omega_y = k) =  a(\rho)^{k}(1-a(\rho)),  \; k\geq 0.
\end{equation}
Define 
\begin{equation}
\label{eq:defalpha}
\nota{\alpha(\rho):=\frac{a(\rho)}{1-a(\rho)}}=\frac{2\rho-1}{1-\rho},
\end{equation}
one easily obtains the identity $\E_{\mu_\rho}(\omega_y)=\alpha(\rho),$ and that $\mu_\rho$ is stationary for $\genzr$.  Further denote by $\nota{\widehat{\mu}_\rho}$ the distribution on $\widehat{\mathcal{E}}_\Z$ (defined in \eqref{eq:EZhat}) such that under $\widehat{\mu}_\rho$, 
\begin{itemize}
\item the zero-range configuration $\omega_{\mid \Z \setminus \{0\}}$ is distributed as $\mu_\rho$, defined by \eqref{eq:Defmurho}, everywhere except at the origin.
\item At the origin, $\omega_0$ is distributed, independently from the rest of the configuration, as 
\begin{equation}
\label{eq:centralcluster}
\widehat{\mu}_\rho(\omega_0=k)=(k+2) a(\rho)^k \rho(1-a(\rho))^2
\end{equation}
\item the distribution of $X_0$ conditionally to $\omega$ is uniform in $\{-\omega_0-1,\dots,0\}$.
\end{itemize}
In other words, for any $(\omega, x)\in \widehat{\mathcal{E}}_\Z$, we have 
\begin{equation}
\label{eq:mumuhat}
\widehat{\mu}_\rho(\omega,x)=\rho(1-a(\rho))\mu_\rho(\omega)
\end{equation}

\medskip

 We claim that the stationary state $\pi_\rho$ for the FEP can be built through the mapping,  by
\begin{equation}
\label{eq:mappingdistributions}
\pi_\rho =\widehat{\mu}_\rho \circ \Pi\qquad \mbox{ and }\qquad \widehat{\mu}_\rho =\pi_\rho \circ \Pi^{-1}.
\end{equation}
Note that through this mapping identity, sites in the zero-range process are associated with \emph{clusters of particles in the FEP,} namely consecutive particles preceded by an empty site.
Given the apparent  breaking of translation invariance in the definition of the central cluster $(\omega_0)$ in $\widehat{\mu}_\rho$, it is not a priori clear that $\widehat{\mu}_\rho \circ \Pi$ would yield a translation invariant distribution.  The core of the argument is that the central particle cluster, the one of the origin, needs the shifted distribution \eqref{eq:centralcluster} to account for the multiple ways it can be placed at the origin, since $X_0\sim Unif(\{-\omega_0-1,\dots,0\})$. The proof of identity \eqref{eq:mappingdistributions} is given in Appendix \ref{sec:statPi}, where we given an alternative construction for $\widetilde{\pi}_ \rho:=\widehat{\mu}_\rho \circ \Pi$, and prove that the latter is translation invariant, and satisfies \eqref{pirho2}.

\medskip

It might also appear counter-intuitive that the stationary state $\pi_\rho$ for the FEP dynamics is mapped to a distribution $\widehat{\mu}_{\rho}$ whose marginal in  $\omega$ is \emph{not stationary} for the mapped zero-range dynamics (due to the loss of translation invariance at the origin in $\omega$). This is to be expected however, because the mapped dynamics must be considered instead on the pair $(\omega,X_0)$, and jumps at the origin have effects on both components, so that the term coming from the displacement of the tagged empty site offsets the loss of stationarity of the zero-range due to the lack of stationarity at the origin. This statement is explored in more detail in Appendix \ref{sec:SSmapping}. 

The core of the arguments amounts to the following : to build $\pi_\rho$, one has to put i.i.d. particle clusters (meaning sets of consecutive particles followed by an empty site) with geometric number of particles distributed as $1+\mu_\rho(\cdot)$, (i.e. $1+\omega$)  and then randomly shift the configuration by choosing the origin (i.e. $X_0$). If the origin is empty, then one can simply build i.i.d. clusters and everything works as expected. If the origin is occupied, however, one needs to choose ``at random'' an occupied site, and choosing an occupied site ``at random'' biases towards larger particle clusters, because the latter contain by definition more particles. In other words, building geometric clusters of particles, and then choosing one particle uniformly among those, the distribution of the cluster of the chosen particle is no longer going to be geometrically distributed, and its distribution is instead going to be the tilted distribution  \eqref{eq:centralcluster}.

\medskip

In what follows, we will use the following result, which can be straightforwardly proved using the zero-range process's attractiveness. 
\begin{proposition}[Discrepancies due to the distortion at the origin]
\label{prop:origindistortion}
There exists a coupling $\Q$ between two trajectories $ \omega$, $\xi$ of the rate $1$ zero-range process such that 
$\omega(0)\sim \mu_\rho$, $\xi(0)\sim \widehat{\mu}_\rho$, and such that 
\[\sum_{y\in\Z} |\xi_y(t)-\omega_y(t)|=|\xi_0(0)-\omega_0(0)|=\xi_0(0)-\omega_0(0) \qquad \Q-\mbox{a.s.}\]
In other words, there exists a coupling between two zero-range processes such that the total number of discrepancies at any time between the two is equal to the number of discrepancies initially at the origin.
\end{proposition}
We do not detail the  proof of this result, it is a direct consequence of our choice of initial distributions, and  the classical so-called \emph{basic coupling}, which we now briefly describe. Because the initial distributions are identical everywhere except at the origin, we can choose $\omega_y(0)=\xi_y(0)$ for any $y\neq 0$. At the origin, we biased the distribution $\widehat{\mu}_\rho$ towards large clusters, we can therefore couple $ \omega$ and $\xi$ in such a way that  $\omega_0(0)\leq \xi_0(0)$. Because the constant rate zero-range process is attractive, we then endow each site with i.i.d. Poisson clock, and each time the clock rings on a site in which both $\omega$ and $\xi$ have particles, a particle performs the same jump in the two configurations. If instead, only $\xi$ has a particle, we make it jump normally in $\xi$ and nothing happens in $\omega$. Under this basic coupling, it is straightforward to show that both $\omega$ and $\xi$ are constant rate zero-range processes, and that Proposition \ref{prop:origindistortion} holds.

\medskip

Because it yields a tight control of discrepancies, Proposition \ref{prop:origindistortion} will allow us to assume that the mapped zero-range process is initially in a stationary state as well rather than $\widehat{\mu}_\rho$, and therefore to prove the mapping estimates we need in the stationary state, as well as use the stationary fluctuations results obtained in \cite{gonccalves2010equilibrium,gonccalves2015stochastic}.

\subsection{Position of the tagged empty site in the dynamic mapping}  Given a trajectory $\omega:=\{\omega(t),\; t\geq 0\}$ of the zero-range process, we denote by $\nota{J^{\rm ZR}_{x,x+1} (t)\in \Z}$ the particle current going through $(x,x+1)$ in $[0,t]$, namely the total particle number going through the edge $(x,x+1)$ before time $t$. It can be formally written as 
\begin{equation}
\label{eq:JtZR}
J^{\rm ZR}_{x,x+1} (t)=\sum_{y\geq x+1} \big\{\omega_y(t)-\omega_y(0)\big\}.
\end{equation}
It is straightforward to show that in the stationary state, its expectation is given by 
\begin{equation}
\label{eq:averagecurrent}
{\bf E}_{\mu_\rho}^N \pa{J^{\rm ZR}_{x,x+1} (t)}= tN^\gamma \mu_{\rho}(\omega_0\geq 1)=tN^\gamma \frac{2\rho-1}{\rho}.
\end{equation}
As before, we denote by $\nota{\bar{J^{\rm ZR}_{x,x+1}} (t):=J^{\rm ZR}_{x,x+1} (t) - tN^\gamma (2\rho-1)/\rho}$ the corresponding centered variable.

\medskip

Consider now a trajectory $\{\eta(t),\;t\geq 0\}$ of the FEP, and $\{\omega(t),\;t\geq 0\}=\Pi^\star[\eta]$ the mapped zero-range trajectory introduced in Section \ref{sec:Dynmap}.  Then, the position $X_0(t)$ of the tagged empty site at time $t$ can then be expressed in terms of the mapped zero-range process $\omega$ through the identity 
\begin{equation}
\label{eq:X0t}
X_0(t)=-J^{\rm ZR}_{-1,0} (t)+ X_0 (0).
\end{equation}
We now give an estimate of the variance of stationary current going through the origin, that will be used repeatedly in the proof of  Theorems \ref{thm:weak_asymFEP} and  \ref{thm:asymFEP}. 
\begin{lemma}
\label{lem:current} 
There exists a constant $C$ independent of $N$ such that
\begin{equation}
\label{eq:L2current}
\E_{\rho}^N \Big[\sup_{0 \leq t \leq T} \bar{J^{\rm ZR}_{-1,0}} (t)^2\Big] \leq C (\log N)^3\Big[N^{\gamma} + \sym N^{4/3}\Big]=O\big(N^{3/2} (\log N)^3\big).
\end{equation}
\end{lemma}

\begin{proof} Because we need an estimate that is uniform in time, the proof of this estimate is rather technical.
 We first get back to the zero-range stationary case. First note that  according to identity \eqref{eq:mappingdistributions},
\[\E_{\rho}^N \Big[\sup_{0 \leq t \leq T} \bar{J^{\rm ZR}_{-1,0}} (t)^2\Big]={\bf E}^N_{\widehat{\mu}_\rho} \Big[\sup_{0 \leq t \leq T} \bar{J^{\rm ZR}_{-1,0}} (t)^2\Big],\]
where as before  ${\bf E}^N_{\nu}$ is the expectation w.r.t. the distribution of the zero-range process started from the initial state $\omega(0)\sim \nu$. The right-hand side above is the expectation of the current going through $(-1,0)$ for a zero-range process started from the state $\widehat{\mu}_\rho$ characterized  in \eqref{eq:mappingdistributions}. According to Proposition \ref{prop:origindistortion} above, since the current discrepancy is less than the total particle discrepancy, we  obtain that
\[\abs{\E_{\rho}^N \Big[\sup_{0 \leq t \leq T}\bar{J^{\rm ZR}_{-1,0}} (t)^2\Big]-{\bf E}_{\mu_\rho} \Big[\sup_{0 \leq t \leq T} \bar{J^{\rm ZR}_{-1,0}} (t)^2\Big]} \leq |{\bf E}^N_{\widehat{\mu}_\rho}(\omega_0(0))-{\bf E}^N_{\mu_\rho}(\omega_0(0))|^{2}:=C(\rho)\]
is at most bounded by a constant, so that it is  enough to prove \eqref{eq:L2current} starting from the zero-range's stationary distribution $ \mu_\rho$.

We now consider the stationary zero-range process $\omega$ initially distributed as $\mu_\rho$. For $\ell > 0$, let
\[G_{\ell}(u) = \Big(1-\frac{u}{\ell}\Big)^+ \mathbf{1}_{\{u \geq 0\}}.\]
Then, 
\begin{multline}
\label{eq:JZRZU}
	\frac{1}{\sqrt{N}} \sum_{x \in \Z} G_{\ell} \big(\tfrac{x}{N}\big) \big[\bar{\omega}_x (t) - \bar{\omega}_x (0)\big] = \frac{1}{\sqrt{N}} \sum_{x \in \Z} G_{\ell} \big(\tfrac{x}{N}\big) \big[\bar{J^{\rm ZR}_{x-1,x}} (t) - \bar{J^{\rm ZR}_{x,x+1}} (t)\big]\\
= \frac{1}{\sqrt{N}} \bar{J^{\rm ZR}_{-1,0}} (t) - \frac{1}{\sqrt{N}} \sum_{x =0}^{N \ell} \frac{1}{N \ell} \bar{J^{\rm ZR}_{x,x+1}} (t),
\end{multline}
where $\bar{J^{\rm ZR}_{x,x+1}} (t)$ was introduced after \eqref{eq:averagecurrent}. Once again, we define 
\[\overline{g}(k)={\bf 1}_{\{k\geq 1\}}-\frac{2\rho-1}{\rho}.\]
For $x\in \Z$, 
\[M_{x,x+1} (t) = \bar{J^{\rm ZR}_{x,x+1}} (t) -  \int_0^t \sym N^2 [\bar{g} (\omega_x (s)) - \bar{g} (\omega_{x+1} (s)] + N^\gamma \bar{g} (\omega_x (s)) ds\]
are independent ${\bf P}^N_{\mu_\rho}$-martingales with quadratic variation
\begin{equation}
\label{eq:quadvar}
\<M_{x,x+1}\>_t = \int_0^t (\sym N^2 + N^\gamma) g(\omega_x (s)) +\sym N^2 g(\omega_{x+1}(s)) ds=O(\sym N^2+N^\gamma).
\end{equation}
In particular, \eqref{eq:JZRZU} rewrites
\begin{multline}\label{eq:jzr}
\bar{J^{\rm ZR}_{-1,0}} (t) = \sum_{x \in \Z} G_{\ell} \big(\tfrac{x}{N}\big) \big[\bar{\omega}_x (t) - \bar{\omega}_x (0)\big]  \\
+ \frac{1}{N \ell} \sum_{x=0}^{N \ell} M_{x,x+1} (t) + \frac{\sym N}{\ell} \int_0^t \big[\bar{g} (\omega_0 (s)) - \bar{g} (\omega_{N\ell+1} (s))\big] ds  + \frac{N^{\gamma-1}}{\ell} \int_0^t \sum_{x=0}^{N\ell}  \bar{g} (\omega_x (s)) ds.
\end{multline}

We first bound the last line in \eqref{eq:jzr}. For the martingale term, by Doob's inequality and \eqref{eq:quadvar}, 
\begin{equation}
\label{eq:1Jzr}
{\bf E}^N_{\mu_\rho} \Big[ \sup_{0 \leq t \leq T} \Big(\frac{1}{N \ell} \sum_{x=0}^{N \ell} M_{x,x+1} (t)\Big)^2\Big] \leq 4  {\bf E}^N_{\mu_\rho} \Big[  \Big(\frac{1}{N \ell} \sum_{x=0}^{N \ell} M_{x,x+1} (T)\Big)^2\Big] \leq C \frac{\sym N+N^{\gamma-1}}{\ell}.
\end{equation}
For the last two terms in \eqref{eq:jzr}, which are time integrals, by Cauchy-Schwarz inequality and since both $\overline{g}(\omega_x(s))$ and $\overline{\omega}_x(s)$ are bounded in $L^2({\bf P}^N_{\mu_\rho})$, 
\begin{equation}
\label{eq:2Jzr}
{\bf E}^N_{\mu_\rho} \Big[ \sup_{0 \leq t \leq T} \Big( \frac{\sym N}{\ell} \int_0^t \big[\bar{g} (\omega_0 (s)) - \bar{g} (\omega_{N\ell+1} (s))\big] ds \Big)^2\Big] \leq CT^2 \frac{\sym N^2}{\ell^2},
\end{equation}
\begin{equation}
\label{eq:3Jzr}
{\bf E}^N_{\mu_\rho} \Big[ \sup_{0 \leq t \leq T} \Big( \frac{N^{\gamma-1}}{\ell} \int_0^t \sum_{x=0}^{N\ell}  \bar{g} (\omega_x (s)) ds \Big)^2\Big] \leq CT^2 \frac{N^{2\gamma-1}}{\ell}.\end{equation}

Now, we bound the first term on the right side of \eqref{eq:jzr}. Clearly, 
\begin{equation}
\label{eq:4Jzr}
{\bf E}^N_{\mu_\rho} \Big[ \Big(  \sum_{x \in \Z} G_{\ell} \big(\tfrac{x}{N}\big)  \bar{\omega}_x (0) \Big)^2\Big] \leq CN\ell.
\end{equation}
It remains to bound
\begin{equation}\label{eq:3_3}
{\bf E}^N_{\mu_\rho} \Big[ \sup_{0 \leq t \leq T} \Big(  \sum_{x \in \Z} G_{\ell} \big(\tfrac{x}{N}\big)  \bar{\omega}_x (t) \Big)^2\Big].
\end{equation}
The main issue to estimate this quantity lies in the time supremum inside the expectation. Without it, one would straightforwardly obtain a bound of order  $O(N \ell)$ since the zero-range process is stationary. Below, we will show that \eqref{eq:3_3} is of order $O \big(N (\log N)^3 \ell \big)$. To this end, we first divide the time interval $[0,T]$ into small intervals of size $\epsilon=\epsilon_N:=N^{-100}$. More precisely, let $t_i = iT\epsilon$  for $0 \leq i \leq 1/\epsilon$.  Note that in order for the set of configurations $\{\omega(t_i), 0\leq i\leq 1/\epsilon\}$ to be different than $\{\omega(t), 0\leq t\leq T\}$, there must have been an interval $[t_i, t_{i+1}[$ in which at least two particle jumps occurred. In particular, we bound \eqref{eq:3_3} by
\begin{equation}\label{eq:3_1}
	{\bf E}^N_{\mu_\rho} \Big[ \sup_{0 \leq i \leq \epsilon^{-1}} \Big(  \sum_{x \in \Z} G_{\ell} \big(\tfrac{x}{N}\big)  \bar{\omega}_x (t_i) \Big)^2\Big] + {\bf E}^N_{\mu_\rho} \Big[ \sup_{0 \leq t \leq T} \Big(  \sum_{x \in \Z} G_{\ell} \big(\tfrac{x}{N}\big)  \bar{\omega}_x (t) \Big)^2 \mathbf{1}_{A_N}\Big],
\end{equation}
where $A_N$ is the event 
\begin{multline}
A_N=\bigcup_{i=0}^{1/\epsilon-1}\Big\{\mbox{at least two particle jumps occurred in the box} \{-1,\dots, N \ell+1\}\\
\mbox{during the time interval }[t_i, t_i+1[\Big\}.
\end{multline}
Since for any fixed $i$, the number of jumps in the box $[-1,N\ell+1]$ during  the time interval $[t_i,t_{i+1}]$ is bounded by a Poisson random variable of parameter $T \epsilon N^{2}(N\ell+1) = O(N^{-97} \ell)$,  we have
\[	{\bf P}^N_{\mu_\rho}  (A_N) \leq C \epsilon^{-1} \big( N^{-97} \ell \big)^2 = O \big( N^{-94} \ell^2\big).\]
By Cauchy-Schwarz inequality, we bound the second term in \eqref{eq:3_1} by
\begin{multline*}
{\bf E}^N_{\mu_\rho} \Big[ \sup_{0 \leq t \leq T} \Big(  \sum_{x \in \Z} G_{\ell} \big(\tfrac{x}{N}\big)  \bar{\omega}_x (t) \Big)^4 \Big]^{1/2} {\bf P}^N_{\mu_\rho} \big(A_N\big)^{1/2} \\
\leq  \Big\{ \sum_{x \in \Z} G_{\ell} \big(\tfrac{x}{N}\big)^4 {\bf E}^N_{\mu_\rho} \Big[ \sup_{0 \leq t \leq T} \Big(   \bar{\omega}_x (t) \Big)^4 \Big] \Big\}^{1/2} {\bf P}^N_{\mu_\rho} \big(A_N\big)^{1/2}.
\end{multline*}
A very rough bound shows that 
\begin{equation}\label{eq:3_2}
{\bf E}^N_{\mu_\rho} \Big[ \sup_{0 \leq t \leq T} \omega_0 (t)^4 \Big] \leq C N^8.
\end{equation}
Indeed, let $N_{0,T}$ be the number of particles that are initially outside of the box $[-2N^2T,2N^2T]$ and visit the origin before time $T$. Then, 
\[{\bf E}^N_{\mu_\rho} \Big[ \sup_{0 \leq t \leq T} \omega_0 (t)^4 \Big] \leq C \Big({\bf E}^N_{\mu_\rho} \Big[ \Big( \sum_{|y| \leq 2 N^2 T} \omega_{y} (0) \Big)^4\Big] + {\bf E}^N_{\mu_\rho} \Big[ N_{0,T}^4 \Big] \Big) \leq C \Big( N^8 + {\bf E}^N_{\mu_\rho} \Big[ N_{0,T}^4 \Big] \Big).\]
Dividing the sets $\Z \backslash [-2N^2T,2N^2T]$ into 
\begin{equation}
\bigcup_{k \geq 2} \big([(-k-1)N^2T,-kN^2T] \cup [kN^2T,(k+1)N^2T]\big),
\end{equation} and since a particle travels at speed at most $N^2$, standard large deviation arguments yield
\[{\bf E}^N_{\mu_\rho} \Big[ N_{0,T}^4 \Big]  \leq C \sum_{k \geq 1} N^8 e^{-CkN^2} \leq C N^8 e^{-CN^2},\]
which proves \eqref{eq:3_2}.  Thus, although this bound is certainly not optimal, the second term in \eqref{eq:3_1} by is finally bounded by $C N^{-40} \ell^2$. 

\medskip

We now turn to the first term in \eqref{eq:3_1}. Define the event
\begin{equation}
B_N = \Big\{ \sup_{0 \leq i \leq 1/\epsilon} \Big| \sum_{x \in \Z} G_{\ell} \big(\tfrac{x}{N}\big)  \bar{\omega}_x (t_i) \Big| > N^{1/2} (\log N)^{3/2} \ell^{1/2}   \Big\}.
\end{equation}
By Cauchy-Schwarz inequality and \eqref{eq:3_2}, we write 
\begin{align*}
{\bf E}^N_{\mu_\rho} \Big[ \sup_{0 \leq i \leq 1/\epsilon} \Big(  \sum_{x \in \Z} G_{\ell} \big(\tfrac{x}{N}\big)  \bar{\omega}_x (t_i) \Big)^2\Big]&\leq N (\log N)^3 \ell + {\bf E}^N_{\mu_\rho} \Big[ \sup_{0 \leq i \leq 1/\epsilon} \Big(  \sum_{x \in \Z} G_{\ell} \big(\tfrac{x}{N}\big)  \bar{\omega}_x (t_i) \Big)^2 \mathbf{1}_{B_N}\Big] \\
&\leq N (\log N)^3 \ell + {\bf E}^N_{\mu_\rho} \Big[ \sup_{0 \leq i \leq 1/\epsilon} \Big(  \sum_{x \in \Z} G_{\ell} \big(\tfrac{x}{N}\big)  \bar{\omega}_x (t_i) \Big)^4\Big]^{1/2}   {\bf P}^N_{\mu_\rho}  (B_N)^{1/2} \\
&\leq N (\log N)^3 \ell + C N^5 \ell {\bf P}^N_{\mu_\rho}  (B_N)^{1/2}.
\end{align*}
Using the stationarity of the zero-range process and a standard  large deviation principle,  by union bound,
\[{\bf P}^N_{\mu_\rho}  (B_N) \leq \frac{C}{\epsilon} e^{-C (\log N)^3}.\] 
Thus, the  first term in \eqref{eq:3_1} is bounded by $C N (\log N)^3 \ell$ for $N$ large enough, which finally yields 
\begin{equation}
\label{eq:5Jzr}
{\bf E}^N_{\mu_\rho} \Big[ \sup_{0 \leq t \leq T} \Big(  \sum_{x \in \Z} G_{\ell} \big(\tfrac{x}{N}\big)  \bar{\omega}_x (t) \Big)^2\Big]=O(N^{-40}\ell^2+N\ell (\log N)^3).
\end{equation}

\medskip

Putting together Equations \eqref{eq:1Jzr}, \eqref{eq:2Jzr}, \eqref{eq:3Jzr}, \eqref{eq:4Jzr}, \eqref{eq:5Jzr}, and by the elementary inequality $(\sum_{k=1}^n a_k)^2\leq n\sum_{k=1}^n a_k^2$,  we obtain,
\[{\bf E}^N_{\mu_\rho} \Big[\sup_{0 \leq t \leq T} \bar{J^{\rm ZR}_{-1,0}} (t)  ^2\Big] \leq C \Big( N^{-40}\ell^2+N (\log N)^3  \ell+ \frac{\sym N+N^{\gamma-1}}{\ell} + \frac{\sym N^2}{\ell^2} + \frac{N^{2\gamma-1}}{\ell}\Big).\]
We then prove the first bound in \eqref{eq:L2current} by taking $\ell=N^{1/3}$ if $\gamma \leq 4/3$ and $\sym=1$, and by taking $\ell=N^{\gamma-1}$ otherwise. For the second bound in \eqref{eq:L2current}, we simply use our assumption that $\gamma\leq 3/2$.
\end{proof}

\begin{remark}
\label{rmk:supremum}
Note that to bound \eqref{eq:3_3}, the main ingredient is a large deviation principle for the sum $\sum_{x \in \Z} G_{\ell} \big(\tfrac{x}{N}\big)  \bar{\omega}_x (t)$, and the time-uniform bound turns out to have an $O((\log N)^3)$ correction compared with the one without the supremum inside the expectation. This technique holds as soon as the quantity to estimate has exponentially decaying tails (in particular when a large deviations principle holds) and result will be used repeatedly in the following section. Since the argument will always be analogous, we will not always detail the time uniform estimate, and instead refer to the proof of Lemma  \ref{lem:current}. To be quite explicit, except for Martingales, for which Doob's inequality yields the wanted result, the scheme to obtain a time-uniform estimate will always be the following
\begin{enumerate}[1.]
\item Divide the time interval into a polynomial number (e.g. $1/\epsilon:=N^{100}$) of small time intervals $[t_i, t_i+1[$. By observing all the $\omega(t_i)$, one actually observes all the $\{\omega(t), \; 0\leq t\leq T\}$, unless two jumps occurred in one of the small time intervals.
\item Since the process in a box of size $K$ jumps at rate $O(KN^2)$, if $K=O(N^2)$, the latter happens with arbitrarily small probability, by letting $\epsilon$ be as small as wanted.
\item Then, the supremum over the $t_i$'s is obtained by stationarity and union bound, which makes us loose a factor $1/\epsilon$, which is not a problem if the probability that the relevant quantity is too large vanishes exponentially in $N$.
\end{enumerate}
\end{remark}

%
%
%
%
%
%
%
%
%

\section{Proof of Theorems \ref{thm:weak_asymFEP} and  \ref{thm:asymFEP}}
\label{sec:mappingproof}
In this section, we now prove Theorems  \ref{thm:weak_asymFEP} and  \ref{thm:asymFEP}, by using the  mapping defined in the previous section.

\subsection{Mapping estimates} 

We need to express the density fluctuation field $\mc{Y}^N_t (G) $ defined in \eqref{eq:YNt} for FEP through that of the zero-range process.  We keep the same notations as in the dynamical mapping defined in Section \ref{sec:Dynmap}, $ (X_y(t))$ is the sequence of successive empty sites in $\eta(t)$ seen from the tagged empty site $X_0(t)$, and $\omega=\Pi^\star[\eta]$ is the resulting constant rate zero-range process. Straightforward computations then yield  
\begin{align}
\label{trans1}
\mc{Y}^N_t (G) &= \frac{1}{\sqrt{N}}  \sum_{x \in \Z} \bar{\eta}_x (t) G \big(\tfrac{x}{N} - t v_N\big)\nonumber\\
&= \frac{1}{\sqrt{N}}  \sum_{y \in \Z} \bigg[ - \rho G \Big(\tfrac{X_y (t)}{N} - t v_N\Big)+\sum_{x=X_y(t)+1}^{X_{y+1} (t)-1} (1 - \rho) G \big(\tfrac{x}{N} - t v_N\big)  \bigg] \nonumber\\
&= \frac{1-\rho}{\sqrt{N}}  \sum_{y \in \Z} G \Big(\tfrac{X_y (t)}{N} - t v_N\Big) \bar{\omega}_y (t) + \frac{1-\rho}{\sqrt{N}}  \sum_{y \in \Z} \sum_{x=X_y(t)+1}^{X_{y+1}(t)-1} \Big[ G \big(\tfrac{x}{N} - t v_N\big) - G \big(\tfrac{X_y (t)}{N} - t v_N\big)\Big],
\end{align}
where according to \eqref{eq:defalpha},
\begin{equation}
\label{eq:omegabar}
\bar{\omega}_y (t):= \omega_y (t)-{\bf E}_{\mu_{\rho}}^N (\omega_0 (0))= \omega_y (t)-\frac{2\rho-1}{1-\rho}.
\end{equation}

\medskip

We first deal with the second term on the right side in \eqref{trans1}, and show that it vanishes in $L^1(\bb{P}_{\rho}^N)$ uniformly in time as  $N\to\infty$. In what follows, we define
\begin{equation}
\label{eq:DefyNvN}
y_N=y_N(T,\rho):=(1-\rho)TNv'_N, \qquad \qquad v'_N=v_N+N^{\gamma-1} \frac{2\rho-1}{\rho}=N^{\gamma-1}\frac{1-\rho}{\rho^2},
\end{equation}
we claim the following.

\begin{proposition}
\label{prop:error}
For any smooth function $G\in \mathcal{S}$ with compact support,
\begin{equation}
\label{eq:lemerror}
\lim_{N\to\infty} \E^N_\rho \Big[ \sup_{0 \leq t \leq T} \Big|  \frac{1}{\sqrt{N}}  \sum_{y \in \Z} \sum_{x=X_y(t)+1}^{X_{y+1}(t)-1} \Big[ G \big(\tfrac{x}{N} - t v_N\big) - G \big(\tfrac{X_y (t)}{N} - t v_N\big)\Big] \Big| \Big]=0.
\end{equation}
\end{proposition}

\begin{proof}
Denote
\[M_G (y,t) = \sup_{X_y (t) \leq x \leq X_{y+1} (t)} |G^\prime \big(\tfrac{x}{N} - t v_N\big)|.\]
By Taylor's expansion, since for any $y$, the number of terms in the sum over $x$ is $\omega_y (t)+1$ (cf. \eqref{eq:omegaXyt})
\begin{multline}
\label{eq:lemerror1}
\E^N_\rho \Big[  \sup_{0 \leq t \leq T} \Big|  \frac{1}{\sqrt{N}}  \sum_{y \in \Z} \sum_{x=X_y(t)+1}^{X_{y+1}(t)-1} \Big[ G \big(\tfrac{x}{N} - t v_N\big) - G \big(\tfrac{X_y (t)}{N} - t v_N\big)\Big] \Big| \Big] \\
\leq  \E^N_\rho \Big[  \sup_{0 \leq t \leq T} \Big|  \frac{1}{N^{3/2}}  \sum_{y \in \Z} (\omega_y (t)+1)^2 M_G (y,t) \Big|\Big].
\end{multline}
The next step is to restrict the support of the sum over $y$. Recall \eqref{eq:DefyNvN} and define
\[\Gamma_N:=\{y\in \Z\; :\; |y-y_N|< N^{4/3} \}, \]
so that 
\begin{multline}
\label{trans2}
\E^N_\rho \Big[ \sup_{0 \leq t \leq T} \Big|  \frac{1}{N^{3/2}}  \sum_{y \in \Z} (\omega_y (t)+1)^2 M_G (y,t) \Big| \Big] \\
\leq  \E^N_\rho \Big[  \sup_{0 \leq t \leq T} \frac{1}{N^{3/2}}  \sum_{y \in \Gamma_N} (\omega_y (t)+1)^2 M_G (y,t) \Big] +  \E^N_\rho \Big[   \sup_{0 \leq t \leq T} \frac{1}{N^{3/2}}   \sum_{y \in\Gamma_N^c} (\omega_y (t)+1)^2 M_G (y,t)  \Big]\\
\leq \frac{  \|G^\prime\|_\infty}{N^{3/2}} \E^N_\rho \Big[  \sup_{0 \leq t \leq T}  \sum_{y \in \Gamma_N} (\omega_y (t)+1)^2 \Big] 
+\frac{1}{N^{3/2}} \sum_{y \in \Gamma_N^c}\E^N_\rho \Big[ \sup_{0 \leq t \leq T} (\omega_y (t)+1)^8 \Big]^{1/4} \E^N_\rho \Big[   \sup_{0 \leq t \leq T} M_G (y,t)^{4/3} \Big]^{3/4},
\end{multline}
by H{\" o}lder inequality.  Note that 
\[\E^N_\rho \Big[   \sum_{y \in \Gamma_N} (\omega_y (t)+1)^2 \Big] \leq C N^{4/3}\]
and that  the sum inside the expectation satisfies a large deviations principle. Following the same steps as in the proof of Lemma \ref{eq:L2current} (see Remark \ref{rmk:supremum}), we obtain that the first term on the right-hand side in \eqref{trans2} is $ O\big(N^{4/3-3/2}(\log N)^3\big)$ and therefore vanishes as $N\to\infty$. We now consider the second sum.  First observe that
\begin{equation}\label{omega_sup}
\E^N_\rho \Big[  \sup_{0 \leq t \leq T} (\omega_0 (t)+1)^8 \Big] \leq \E^N_\rho \Big[  \sup_{0 \leq t \leq T}  \sum_{|y| \leq (\log N)^2} (\omega_y (t)+1)^8 \Big] \leq C (\log N)^8,
\end{equation}
where in the second inequality we use Remark \ref{rmk:supremum} again. Thus, the second term on the right-hand side in \eqref{trans2} is bounded by
\[\frac{C (\log N)^2}{N^{3/2}} \sum_{y \in \Gamma_N^c}  \E^N_\rho \Big[   \sup_{0 \leq t \leq T} M_G (y,t)^{4/3} \Big]^{3/4}.\]
Denote by $ \nota{A:=A(G)>0}$ the size of $G$'s support, meaning that $G$ vanishes on  the set $\R\setminus [-A,A]$. Since  $M_G (y,t) = 0$ if $|\tfrac{X_y (t)}{N} - t v_N| > A$ and $|\tfrac{X_{y+1} (t)}{N} - t v_N| > A$, using \eqref{eq:lemerror1} and a union bound, we obtain that the expectation in \eqref{eq:lemerror} is  bounded from above by 
\begin{equation}
\label{eq:errormain}
\frac{C(\rho,G)(\log N)^2}{N^{3/2}} \sum_{y\mbox{ \tiny or }y-1 \in \Gamma_N^c}  \P^N_\rho \big( \inf_{0 \leq t \leq T} |\tfrac{X_y (t)}{N} - t v_N| \leq A \big)^{3/4}+O(N^{-1/6}(\log N)^3) 
\end{equation}
We now need to show that for any $y\notin \Gamma_N$, $\tfrac{X_y(t)}{N} - v_Nt$ is  w.h.p. outside of the support of $G$.

\medskip

Fix $y\in \Z$, using \eqref{eq:omegaXyt} and \eqref{eq:X0t}, we obtain
\begin{equation}
\label{eq:DefXy}
X_y(t) = \sum_{y^\prime = 1}^y [\omega_{y'-1} (t) + 2]  + X_0 (t) =  \sum_{y^\prime = 1}^y [\omega_{y'-1} (t) + 2]  -J^{\rm ZR}_{-1,0} (t) + X_0 (0).
\end{equation}
Recall from \eqref{eq:DefyNvN} the definition of $v_N'$.. Using \eqref{eq:averagecurrent} and \eqref{eq:omegabar}, we therefore write,  
\begin{align*}
	\tfrac{X_y(t)}{N} - v_Nt&=\frac{1}{N}\cro{ \sum_{y^\prime = 1}^y \omega_{y'-1} (t) + 2y -J^{\rm ZR}_{-1,0} (t) + X_0 (0)}-v_Nt \\
	&=\frac{1}{N}\cro{ \sum_{y^\prime = 1}^y \overline{\omega}_{y'-1} (t) -\overline{J^{\rm ZR}_{-1,0}} (t)}-t v'_N +\frac{2y+X_0(0)}{N}+\frac{y}{N}\frac{2\rho-1}{1-\rho}\\
	&=\frac{1}{N}\cro{ \sum_{y^\prime = 1}^y \overline{\omega}_{y'-1} (t) -\overline{J^{\rm ZR}_{-1,0}} (t)+X_0(0)+\frac{y-y_N}{(1-\rho)}}. 
\end{align*}
For $y\in\Gamma_N^c$, shorten $\nota{\bar y:= |y-y_N|/(1-\rho)}\geq N^{4/3}$. Since $\rho\in [0,1]$, in order to have  
\[\inf_{0 \leq t \leq T} \big| \tfrac{X_y(t)}{N} - v_Nt \big|\leq A,\]
we must have either 
\begin{equation}
	\label{eq:trio}
	X_0(0)\geq \frac{\bar y-NA}{3},\quad \sup_{0 \leq t \leq T} \abs{\sum_{y^\prime = 1}^y \overline{\omega}_{y'-1} (t)}\geq \frac{\bar y-NA}{3} \quad \mbox{ or }\quad \sup_{0 \leq t \leq T}\abs{\overline{J^{\rm ZR}_{-1,0}} (t)}\geq \frac{\bar y-NA}{3}.
\end{equation}
By a standard large deviations estimate, the first event occurs with exponentially small probability under $\bb{P}_\rho^N$, which we denote by  
\begin{equation}
\label{eq:p1estimate}
p_y:= \P_\rho^N \pa{X_0(0)\geq \frac{\bar y-NA}{3}}=O(e^{-c(\bar y -NA)})
\end{equation}
for some positive constant $c$ independant of $y$. In what follows, the constant $c$ can change from line to line. Under the product geometric distribution $\mu_\rho$ defined in \eqref{eq:Defmurho},  a large deviations estimate also yields that as $y\to + \infty$
\begin{equation}
\label{eq:LDEmu}
\mu_\rho\pa{\frac{1}{y}\sum_{y'=1}^y\overline{\omega}_{y'-1}\geq x}=O(e^{-cyx^2} ),
\end{equation}
for some positive constant $c$. Note that because of the non-stationarity at the origin, $\omega(t)$ is not a stationary Facilitated Zero-Range process. However, thanks to Proposition \ref{prop:origindistortion}, the equilibrium large deviations estimate \eqref{eq:LDEmu} yields 
\begin{multline}
\label{eq:qy}
q_y:=\bb{P}_{\rho}^N\pa{\sup_{0 \leq t \leq T} \abs{\sum_{y^\prime = 1}^y \overline{\omega}_{y'-1} (t)}\geq \frac{\bar y-NA}{3}}\\
={\bf P}_{\mu_\rho}^N\pa{\sup_{0 \leq t \leq T} \abs{\sum_{y^\prime = 1}^y \overline{\omega}_{y'-1} (t)} \geq \frac{\bar y-NA}{3}}
=O\big(N^4 y^6 e^{-c(\bar y -NA)^2/|y|} +  y^{-4}\big),
\end{multline}
where in the last step, we divide the time interval $[0,T]$ into $N^{4}y^6$ small intervals and use the same argument as in Lemma \ref{lem:current} (cf. Remark \ref{rmk:supremum}).
We split the exponent depending on whether $y\in B_N:=\Gamma_N^c \cap \{-2y_N,\dots,2y_N\}$ or $y\in \Gamma_N^c \setminus B_N$. If $y\in B_N$ 
 \[\frac{(\bar y -NA)^2}{|y|}\geq \frac{( N^{4/3}-NA)^2}{2y_N}\geq cN^{8/3-\gamma}\geq c N\]
 since $\gamma$ is assumed to be less than $3/2$. If instead, $y\in \Gamma_N^c \setminus B_N$, we have $\bar y/|y|=1-y_N/y\geq 1/2$, so that the probability above is of order $O(e^{-c \bar y})$. Putting those to statement together yields 
\begin{equation}
\label{eq:p2estimate}
q_y\leq N^{4+6\gamma} e^{-cN}{\bf 1}_{\{y\in B_N\}}+ N^4 y^6 e^{-c\bar{y}}{\bf 1}_{\{y\in \Gamma_N^c\setminus B_N\}} + y^{-4}.
\end{equation}
Finally, to estimate the probability of the last event in \eqref{eq:trio}, we use Lemma \ref{lem:current} and Chebychev's inequality, to get
\begin{equation}
\label{eq:p3estimate}
r_y:=\bb{P}_{\rho}^N\pa{\sup_{0 \leq t \leq T} \Big| \overline{J^{\rm ZR}_{-1,0}} (t) \Big| \geq\frac{\bar y-NA}{3} } \leq \frac{c N^{3/2} (\log N)^3}{(\bar y -NA)^2}.
\end{equation}
Since
\begin{equation}
\label{eq:calcPy}
\P^N_\rho \big(\inf_{0 \leq t \leq T} |\tfrac{X_y (t)}{N} - t v_N| \leq A \big)\leq p_y+q_y+r_y,
\end{equation}
putting \eqref{eq:p1estimate}, \eqref{eq:p2estimate} and \eqref{eq:p3estimate}, we obtain that  for any $y\in \Gamma_N^c$
\begin{equation*}
\P^N_\rho \big(\inf_{0 \leq t \leq T} |\tfrac{X_y (t)}{N} - t v_N| \leq A \big)^{3/4}\leq \frac{c N^{9/8} (\log N)^{9/4}}{\bar y^{3/2}}.
\end{equation*}
Summing this identity over $y\in \Gamma_N^c$ and multiplying by $ N^{-3/2}$ yields that the first term in \eqref{eq:errormain} vanishes as $N\to\infty$ as wanted, which proves the Lemma.
\end{proof}

Now, we deal with the first term on the right side of \eqref{trans1}, in which we want to replace 
\begin{equation}
\label{eq:repXy}
X_y (t) \qquad \mbox{ by } \qquad \frac{y}{1-\rho} - tN^{\gamma} \frac{2\rho-1}{\rho},
\end{equation}
where  the second term represents the mean displacement of a tagged empty site in a time $t$. Recall from \eqref{eq:DefyNvN} that
\[v'_N=v_N+N^{\gamma-1} \frac{2\rho-1}{\rho}=N^{\gamma-1}\frac{1-\rho}{\rho^2},\]
which represents the macroscopic velocity of a given particle \emph{relative to the position of a tagged empty site.} 
Lemma \ref{lem:replaceX_y} below justifies replacement \eqref{eq:repXy}, and proves, together with \eqref{trans1} and Proposition \ref{prop:error}, that  for any $\varepsilon > 0$,
\begin{align}
\label{trans3}
\limsup_{N\to\infty}\P_{\rho}^N \Big( \sup_{0 \leq t \leq T} \Big| \mc{Y}^N_t (G) - \frac{1-\rho}{\sqrt{N}}  \sum_{y \in \Z} G \Big(\frac{y}{N(1-\rho)} - t v'_N\Big) \bar{\omega}_y (t) \Big| > \varepsilon \Big) =0.
\end{align}

We now state and prove the replacement \eqref{eq:repXy}.
\begin{lemma}
\label{lem:replaceX_y} 
For any test function $G$ with compact support,  for any $\varepsilon > 0$,
\begin{equation}
	\label{eqn:replacementXy}
	\lim_{N\to\infty} \P_{\rho}^N \Big[ \sup_{0 \leq t \leq T} \Big| \frac{1}{\sqrt{N}} \sum_{y \in \Z} \bar{\omega}_y (t) \big[G \big(\tfrac{X_y(t)}{N} - v_Nt\big) - G \big(\tfrac{y}{N(1-\rho)} -  t v'_N\big) \big] \Big| > \varepsilon \Big]= 0.
\end{equation}

\end{lemma}

Note that the sum inside the probability can involve an infinite number of non zero contributions even for a function $G$ with compact support. To overcome this difficulty, fix a smooth compactly supported test function $G\in\mathcal{S}$, and as before let $ \nota{A>0}$ be the size of its support. Recall $y_N=(1-\rho) T Nv'_N$. Define 
 \[\Gamma_N(G):=\{y_N-2AN(1-\rho),\dots, y_N+ 2AN(1-\rho)\}.\]
Then, the above lemma follows immediately from the following two results.

\begin{lemma}
\label{lem:replaceX_y2}
For any test function $G$ with compact support,
\begin{equation}\label{eqn:repXy2}
\lim_{N\to\infty} \E_{\rho}^N \Big[ \sup_{0 \leq t \leq T} \Big| \frac{1}{\sqrt{N}} \sum_{y \in \Gamma_N(G)} \bar{\omega}_y (t) \big[G \big(\tfrac{X_y(t)}{N} - v_Nt\big) - G \big(\tfrac{y}{N(1-\rho)} -  t v'_N\big) \big] \Big|\Big]= 0.
\end{equation}
\end{lemma}

\begin{lemma}
\label{lem:replaceX_y1}
For any test function $G$ with compact support,
\begin{equation}
R_{N,T}(G):= \sup_{0 \leq t \leq T} \Big| \frac{1}{\sqrt{N}} \sum_{y \notin \Gamma_N(G)} \bar{\omega}_y (t) \big[G \big(\tfrac{X_y(t)}{N} - v_Nt\big) - G \big(\tfrac{y}{N(1-\rho)} -  t v'_N\big) \big] \Big|
\end{equation}
vanishes in probability as $N\to\infty$. 
\end{lemma}

We now prove Lemmas \ref{lem:replaceX_y2} and \ref{lem:replaceX_y1}.

\begin{proof}[Proof of Lemma \ref{lem:replaceX_y2}]
For $y\in \Gamma_N(G)$, and $\ell \geq 1$, define 
\begin{equation*}
\Delta_y^\ell (G) :=	G \big(\tfrac{X_y(t)}{N} - tv_N\big) - G \big(\tfrac{y}{N(1-\rho)} -  t v'_N\big) - \frac{1}{2\ell+1} \sum_{|z-y|\leq \ell} \big[G \big(\tfrac{X_z(t)}{N} - tv_N\big) - G \big(\tfrac{z}{N(1-\rho)} -  t v'_N\big) \big].
\end{equation*}
Summing by parts, we bound the absolute value in \eqref{eqn:repXy2}  by
\begin{equation*}
\Big| \frac{1}{\sqrt{N}} \sum_{y \in \Gamma_N(G)} \bar{\omega}_y (t) \Delta_y^\ell (G)  \Big|+ \Big| \frac{1}{\sqrt{N}} \sum_{y \in \Gamma_N(G)} \bar{\omega}^\ell_y (t) \big[G \big(\tfrac{X_y(t)}{N} - tv_N\big) - G \big(\tfrac{y}{N(1-\rho)} -  t v'_N\big) \big] \Big|,
\end{equation*}
where  $\nota{\bar{\omega}^\ell_y (t)=(2\ell+1)^{-1}\sum_{|z-y|\leq \ell}\bar{\omega}_z (t)}$. By \eqref{omega_sup} and Remark \ref{rmk:supremum},  if $\ell > N^\delta$ for some $\delta > 0$, then there exists some constant $C=C(\rho)$ such that
\[\E_{\rho}^N \big[ \sup_{0 \leq t \leq T} \bar{\omega}_y (t)^2 \big]^{1/2} \leq C (\log N)^4 \quad  \mbox{ and }\quad \E_{\rho}^N \big[\sup_{0 \leq t \leq T} \bar{\omega}^\ell_y (t)^2\big]^{1/2}\leq C (\log \ell)^3 /\sqrt{\ell},\]
we obtain that the expectation in \eqref{eqn:repXy2} is bounded from above by triangular and Cauchy-Schwarz inequality by
\begin{multline}
\label{re1-3}
 \frac{C (\log N)^4}{\sqrt{N}} \sum_{y \in \Gamma_N(G)}\E_{\rho}^N \big[\sup_{0 \leq t \leq T} \Delta_y^\ell (G)^2 \big]^{1/2}  \\
 +  \frac{C (\log \ell)^3}{\sqrt{N\ell}} \sum_{y \in \Gamma_N(G)} \E_{\rho}^N \Big[ \sup_{0 \leq t \leq T} \Big( G \big(\tfrac{X_y(t)}{N} - tv_N\big) - G \big(\tfrac{y}{N(1-\rho)} -  t v'_N\big) \Big)^2 \Big]^{1/2}.
\end{multline}
To estimate the first term, rewrite
\begin{multline*}
\Delta_y^\ell (G) =	\frac{1}{2\ell+1} \sum_{|z-y|\leq \ell} \big[G \big(\tfrac{X_y(t)}{N} - v_Nt\big)-G \big(\tfrac{X_z(t)}{N} - v_Nt\big)\Big]\\
 -\frac{1}{2\ell+1} \sum_{|z-y|\leq \ell} \big[G \big(\tfrac{y}{N(1-\rho)} -  t v'_N\big)- G \big(\tfrac{z}{N(1-\rho)} -  t v'_N\big) \big].
\end{multline*}
The second line above is a discrete laplacian, and is therefore of order $O(\ell^2/N^2)$. To estimate the first line, we shorten $U_y:=\tfrac{X_y(t)}{N} - v_Nt$, and develop $G(U_z)$ around $U_y$ to obtain by translation and time invariance, and using the elementary inequality $(a+b+c)^2\leq 3a^2+3b^2+3b^2$
\begin{multline}
\label{eq:Deltayell}
\E_{\rho}^N \big[\sup_{0 \leq t \leq T} \Delta_y^\ell (G)^2 \big]=\frac{3}{N^2(2\ell+1)^2}\E_{\rho}^N  \pa{ \sup_{0 \leq t \leq T} G'(U_y)^2\pa{ \sum_{z=-\ell}^\ell[X_y(t)-X_{y+z}(t)]}^2}\\
+O\pa{\frac{1}{N^4}\E_{\rho}^N \pa{ \sup_{0 \leq t \leq T} [X_0(t)-X_{\ell}(t)]^4}}+O(\ell^4/N^4).
\end{multline}
Rewrite  
\[X_y(t)-X_z(t)=\sum_{y^\prime = z+1}^y [\omega_{y'-1} (t)+2],\]
where the $\omega_{y'-1} (t)$ are i.i.d. geometric variables,  so that by Remark \ref{rmk:supremum}, for $\ell > (\log N)^2$
\[\E_{\rho}^N \pa{ \sup_{0\leq t\leq T}[X_0(t)-X_{\ell}(t)]^4}=O(\ell^4 (\log \ell)^3).\]
In the first term in the right-hand side \eqref{eq:Deltayell}, we rewrite the sum as
\[\sum_{z= -\ell}^{\ell}[X_y(t)-X_{y+z}(t)]=-\sum_{z=1}^\ell  (\ell+1-z) \pa{\omega_{y+z-1} (t)-\omega_{y-z}(t)}.\]
In the right hand side, terms for $z\neq z'$ are independent, mean-$0$ variables, so that by Remark \ref{rmk:supremum},
\begin{multline*}
\E_{\rho}^N \pa{ \sup_{0 \leq t \leq T} G'(U_y)^2 \pa{ \sum_{z=-\ell}^\ell[X_y(t)-X_{y+z}(t)]}^2} \\
\leq C(G) \E_{\rho}^N \pa{ \sup_{0 \leq t \leq T} \pa{ \sum_{z=-\ell}^\ell[X_y(t)-X_{y+z}(t)]}^2}=O\big(\ell^3 (\log \ell)^3\big).
\end{multline*}
Combining these bounds yield
\begin{equation}
\label{eq:Deltayl}
\E_{\rho}^N (\sup_{0 \leq t \leq T} \Delta_y^\ell (G)^2)=O\pa{\frac{\ell (\log \ell)^3}{N^2}+\frac{\ell^4 (\log \ell)^3}{N^4}}.
\end{equation}

We now deal with the second term in \eqref{re1-3}, and write
\begin{multline*}
	 \big|G \big(\tfrac{X_y(t)}{N} - v_Nt\big) - G \big(\tfrac{y}{N(1-\rho)} -  t v'_N\big) \big| \\
	 \leq \|G^\prime\|_\infty \Big| \frac{X_y(t)}{N} - \frac{y}{N(1-\rho)} +  tN^{\gamma-1} \tfrac{2\rho-1}{\rho} \Big|=  N^{-1} \|G^\prime\|_\infty \Big|\sum_{y^\prime = 1}^y \bar{\omega}_{y'-1} (t)   - \bar{J^{\rm ZR}_{-1,0}} (t) + X_0 (0)\Big|.
\end{multline*}
Since, by Remark \ref{rmk:supremum}, \[\sup_{y\in \Gamma_N(G)}\E_{\rho}^N \pa{\sup_{0 \leq t \leq T} \pa{\sum_{y^\prime = 1}^y \bar{\omega}_{y'-1} (t)}^2}=O\big(N^{3/2} (\log N)^3\big),\] 
and $\E_{\rho}^N (X_0(0)^2)$ is a constant, we can bound the second term in \eqref{re1-3} by 
\[\frac{C(G, \rho) (\log \ell)^3}{\sqrt{N\ell}} \Big\{ N^{3/4} (\log N)^{3/2}+ \sqrt{\E_{\rho}^N \big[\sup_{0 \leq t \leq T} \bar{J^{\rm ZR}_{-1,0}} (t)^2\big]} \Big\}=O\pa{\frac{N^{1/4}(\log N)^{3/2} (\log \ell)^3}{\ell^{1/2}}},\]
according to Lemma \ref{lem:current}. 
Adding up this estimate to \eqref{eq:Deltayl}, we bound \eqref{re1-3} and therefore \eqref{eqn:repXy2} by
\[C (\log N)^4 (\log \ell)^3\Big( \frac{\ell^{1/2}}{N^{1/2}} + \frac{\ell^2}{N^{3/2}} + \frac{N^{1/4}}{\ell^{1/2}}\Big),\]
so that  choosing $\ell= N^{5/8}$ proves  the Lemma.
\end{proof}

\begin{proof}[Proof of Lemma \ref{lem:replaceX_y1}]Recall that $A$ is the size of the support of $G$, therefore 
\[y\notin \Gamma_N(G)\quad \Longrightarrow \quad   G \big(\tfrac{y}{N(1-\rho)} -  t v'_N\big) =0,\]
so that the second part of $R_{N,t}(G)$ vanishes, and we can write
\begin{equation*}
\bb{P}_{\rho}^N (|R_{N,T}(G)|>0)\leq\bb{P}_{\rho}^N \pa{\exists y\notin \Gamma_N(G) \mbox{ such that } \inf_{0 \leq t \leq T} |\tfrac{X_y(t)}{N} - v_Nt|\leq A}
\end{equation*} 
Since the positions of the empty sites are ordered in $y$, letting 
\[y_1=y_{1,N}:=y_N-2AN(1-\rho)\quad  \mbox{ and } \quad y_2=y_{2,N}:=y_N+2AN(1-\rho),\]
we rewrite
\begin{equation*}
\bb{P}_{\rho}^N (|R_{N,T}(G)|>0)\leq\bb{P}_{\rho}^N (\inf_{0 \leq t \leq T} |\tfrac{X_{y_1}(t)}{N} - v_Nt|\leq A)+\bb{P}_{\rho}^N (\inf_{0 \leq t \leq T} |\tfrac{X_{y_2}(t)}{N} - v_Nt|\leq A)
\end{equation*} 
Using the same notations as in \eqref{eq:p1estimate}, \eqref{eq:qy} and \eqref{eq:p3estimate}, we have  $\bar y_1=\bar y_2=2 AN$, therefore \eqref{eq:calcPy} yields that $\bb{P}_{\rho}^N (|R_{N,T}(G)|>0)$ vanishes as $N\to\infty$, and in particular $R_{N,T}(G)$ vanishes in probability as wanted. 
\end{proof}

\subsection{Proof of Theorem {\ref{thm:weak_asymFEP}} and {\ref{thm:asymFEP}}} We define 
\[\Phi (\alpha)=\frac{\alpha}{1+\alpha},\]
recall the definition \eqref{eq:DefyNvN} of $v_N'$, and that $\nota{\alpha:=\alpha(\rho)}$ the density of the zero-range process defined in \eqref{eq:defalpha}. Straightforward computations yield
\[v'_N= \frac{N^{\gamma-1}}{1-\rho} \Phi^\prime (\alpha)\]
We introduce the density fluctuation field of the zero-range process, namely 
\begin{equation}\label{ynt}
\mc{Z}^N_t (G) = \frac{1}{\sqrt{N}} \sum_{y \in \Z} \bar{\omega}_y (t) G \big(\tfrac{y}{N} - t v'_N(1-\rho)\big)=\frac{1}{\sqrt{N}} \sum_{y \in \Z} \bar{\omega}_y (t) G \big(\tfrac{y}{N} - tN^{\gamma-1} \Phi'(\alpha)\big),
\end{equation}
and given a test function $G$, define 
\[\widetilde{G}_\rho(u):=(1-\rho)G\pa{\frac{u}{1-\rho}} \qquad\mbox{ and } \qquad \widehat{G}_\rho(u):=\frac{1}{(1-\rho)}G(u(1-\rho)),\]
so that $\widehat{(\widetilde{G}_\rho)}_\rho=G$. So far, we have shown in \eqref{trans3} that for any $\varepsilon \geq 0$,
\begin{equation}
\label{ynt}
\limsup_{N\to\infty}\P_\rho^N(\sup_{0 \leq t \leq T} |\mc{Y}^N_t (G) -\mc{Z}^N_t (\widetilde{G}_\rho)| >\varepsilon)=0.
\end{equation}

We now conclude the proof of  Theorem \ref{thm:weak_asymFEP} and \ref{thm:asymFEP} by stating the following results, which were extracted from \cite{gonccalves2015stochastic} and \cite{gonccalves2010equilibrium}.
\begin{theorem}[{\cite[Proposition 2.1 and Theorem 2.2]{gonccalves2015stochastic}}]
\label{thm:fluc_zrp}
Consider a random distribution $Z_0\in \mathcal{S}'$ with covariance 
\begin{equation}
\label{eq:Z0}
\E(Z_0(G)Z_0(H))=\widetilde{\chi} \langle G,H\rangle,
\end{equation}
where 
\[\widetilde{\chi}(\alpha)=\frac{\Phi'(\alpha)}{\Phi(\alpha)}=\frac{\rho(2\rho-1)}{(1-\rho)^2}\]
is the zero-range process's compressibility. Then, for $\sym=1$,  the zero-range fluctuation process $\{\mc{Z}^N_t, \;0\leq t\leq T\}$ converges in the uniform topology on $D\big([0,T],\mathcal{S}'\big)$ to a  process $\{Z_t, \;0\leq t\leq T\}$ with initial state characterized by \eqref{eq:Z0}, which is the solution 
\begin{enumerate}[(i)]
	\item of the stochastic heat equation
\begin{equation}
\label{eq:DefSHEZR}	
\partial_t Z_t = \Phi' \big(\alpha\big)  \partial_u^2 Z_t  + \sqrt{2 \Phi \big(\alpha\big)} \partial_u \dot{ \mathcal{W}}_t.
\end{equation}
	 in the sense of Proposition \ref{prop:carsolSHE} with $D=\Phi(\alpha)$, $\sigma=\Phi'(\alpha)$, and  $\chi=\widetilde{\chi}(\alpha)$ for $ \gamma < 3/2$.
	\item  of the stochastic Burgers equation
	\begin{equation}
	\label{sbe_zrp2}
	\partial_t Z_t = \Phi'(\alpha)(\rho)  \partial_u^2 Z_t  + \frac{1 }{2} \Phi''(\alpha)\partial_u Z_t^2+ \sqrt{2\Phi(\alpha)} \partial_u \dot{ \mathcal{W}}_t.
\end{equation}
	in the sense of Definition \ref{def:solSBE} with $\chi =\widetilde{\chi}(\alpha)$ for $\gamma = 3/2$.
\end{enumerate}
\end{theorem}

\begin{theorem}[{\cite[Theorem 2.5]{gonccalves2010equilibrium}}]
\label{thm:Gonc2}
For $ \gamma < 4/3$, and $\sym=0$, the zero-range fluctuation process $\{\mc{Z}^N_t, \;0\leq t\leq T\}$ converges weakly in $D([0,T],\mathcal{S}^\prime)$ to the stationary Gaussian process $\{Z_t, \;0\leq t\leq T\}$  in $C([0,T],\mathcal{S}^\prime)$ with mean zero and covariance given by
\[\E [Z_t (G) Z_s (H)] =\widetilde{\chi}  \<G,H\>,\]
for any $s,t \geq 0$ and $G,H \in \mathcal{S}$.
\end{theorem}

\begin{remark}
In \cite{gonccalves2010equilibrium}, the authors considered the fluctuation fields in a different space $\mathcal{H}_{-k}$ instead of $\mathcal{S}^\prime$. However, the choices of the spaces are only relevant when proving tightness of the fluctuation fields, and the readers could check directly that the fluctuation fields are also tight in the space $D([0,T],\mathcal{S}^\prime)$ with respect to the weak uniform topology.
\end{remark}

\begin{remark}
Although the results in \cite{gonccalves2010equilibrium} are only stated for $\gamma \geq 0$, they are obvious true for $\gamma < 0$ since in this case the evolution of the dynamics is too weak to affect the macroscopic behavior of the process.
\end{remark}

Clearly, a process $\{Y_t,\;0\leq t \leq T \}$ is solution of \eqref{eq:DefSHE} (resp. \eqref{sbe_zrp}) iff  the process defined by $Z_t(G):=Y_t(\widehat{G}_\rho)$ is solution of \eqref{eq:DefSHEZR} (resp. \eqref{sbe_zrp2}). Thanks to \eqref{ynt},  Theorem \ref{thm:weak_asymFEP} is therefore a direct consequence of Theorem \ref{thm:fluc_zrp}.

Similarly, Theorem \ref{thm:asymFEP} is a direct consequence of Theorem \ref{thm:Gonc2}.

%
%
%
%
%

\section{Proof by sharp estimates}
\label{sec:proofEqui}
In this section, we give an alternative proof of Theorem \ref{thm:weak_asymFEP} which does not rely on the mapping to the zero-range process. Doing so, we obtain a sharper estimate (in $O(\log^2 \ell /\ell)$ on the equivalence of ensembles for the FEP than the one previously obtained in \cite{blondel2020hydrodynamic} (in $O(\ell^{-1/4})$) to derive the supercritical hydrodynamic limit. This estimate is the main argument needed to derive equilibrium fluctuations, and we feel it is interesting on its own, which is the reason why we give this alternative proof.

\begin{theorem}[Symmetric case]
\label{thm:symFEP}
For $\sym=1$, $\gamma=-\infty$, the FEP's  fluctuation field 
\begin{equation}
\label{eq:YNtsym}
\mc{Y}^N_t (G) = \frac{1}{\sqrt{N}} \sum_{x \in \Z} \bar{\eta}_x (t) G(x/N),
\end{equation}
 converges in the weak uniform topology on $D\big([0,T], \mathcal{S}'\big)$, as $N\to\infty$ to a process $\{Y_t, \; 0\leq t\leq T\} $, which is solution to the stochastic heat equation in the sense of Proposition \ref{prop:carsolSHE}, meaning that  for any $G \in \mathcal{S}$,
\[M_t (G) := Y_t (G) - Y_0 (G) -  D(\rho) \int_0^t Y_s (\partial_u^2 G) ds\]
\[N_t(G) :=\big[M_t (G)\big]^2 - 2 t   \sigma(\rho) \|\partial_u G\|_{L^2 (\R)}^2\]
are both integrable martingales w.r.t. $Y$'s natural filtration, and that for any $t\geq 0$, 
	\begin{equation}
	\label{eq:fixedtimecovariances}
	\E(Y_t(G)Y_t(H))=\chi(\rho) \langle G,H\rangle.
	\end{equation}
\end{theorem}

\medskip
First note that in the stationary state, direct calculation immediately yields \eqref{eq:fixedtimecovariances}. Furthermore, for any compactly supported test functions $G: \R\to\R$, by Dynkin's martingale formula,
\[\mc{M}^N_t (G):= \mathcal{Y}^N_t(G) - \mathcal{Y}^N_0 (G) - \int_0^t \genex \mathcal{Y}^N_s(G) ds\]
is a martingale with quadratic variation given by
\[\<\mc{M}^N_{\cdot} (G)\>_t = \int_0^t \genex \mathcal{Y}^N_s(G)^2 - 2  \mathcal{Y}^N_s(G) \genex \mathcal{Y}^N_s(G) ds.\]
Direct calculations yield
\[\genex \mathcal{Y}^N_s(G) = \frac{1}{\sqrt{N}} \sum_{x \in \Z} [\tau_x h(\eta(s)) - a(\rho)] \partial_u^{2,N} G (x/N), \]
where $h(\eta) = \eta_{-1}\eta_0 + \eta_0\eta_{1} - \eta_{-1} \eta_0 \eta_1$, $a(\rho)$ is $h$'s average under $\pi_\rho$ defined in \eqref{eq:arho} and $\partial_u^{2,N} G$ is a discrete approximation of $G$'s laplacian,
\[ \partial_u^{2,N} G(x/N)=N^2\cro{G((x+1)/N)+G((x-1)/N)-2G(x/N)}.\] 
Furthermore, 
\[\<\mc{M}^N_{\cdot} (G)\>_t = \int_0^t \frac{1}{N} \sum_{x \in \Z} c_{x,x+1} (\eta_s) [\partial_u^N G(x/N)]^2 ds.\]
where similarly
\[ \partial_u^N G(x/N)=N\cro{G((x+1)/N)-G(x/N)}.\]

Note that by the Cauchy-Schwarz inequality and from the exponential decay of correlations under the measure $\pi_\rho$, one straightforwardly obtains the following result : 
\begin{lemma}\label{lem:qv}
	For any compactly supported test functions $G: \R\to\R$,
	\[\lim_{N\to\infty} \<\mc{M}^N_{\cdot} (G)\>_t = t\sigma(\rho) ||\partial_u G||_{L^2 (\R)} \quad \text{in $L^2 (\pi_\rho)$},\]
where $\sigma(\rho)$ is defined in \eqref{eq:Defsigmarho}.
\end{lemma}

Indeed, to prove the above result, one just note that
\[\lim_{N\to\infty} \E_{{\rho}}^N \big[  \<\mc{M}^N_{\cdot} (G)\>_t  \big] = t\sigma(\rho) ||\partial_u G||_{L^2 (\R)}\]
and the variance of $\<\mc{M}^N_{\cdot} (G)\>_t$ is bounded by a constant multiple of
\[\frac{t^2}{N^2}  \sum_{x \in \Z} [\partial_u^N G(x/N)]^4,\]
which has order $O(N^{-1})$ and thus vanishes as $N \rightarrow \infty$. Following classical estimates (see \cite{zhao2023stationary} for example), it is not hard to show the following result.
\begin{lemma}\label{lem:tightness}
	The sequence $\{\mathcal{Y}^N_t, 0 \leq t \leq T\}$ is tight with respect to the weak uniform topology of $D([0,T],\mathcal{S}^\prime)$.
\end{lemma}

The main ingredient to derive the equilibrium fluctuations for the FEP is the Boltzmann-Gibbs principle, whose proof will be adapted from \cite{klscaling}. For any local function $\psi: \Sigma \rightarrow \R$, denote
\[\widetilde{\psi}(\rho) = E_{\pi_\rho} [\psi].\]

\begin{proposition}[Boltzmann-Gibbs principle]
\label{thm:bg}
	For any local function of the configuration $\psi$, any compactly supported smooth function $G$, and any $t>0$, we have 
	\begin{equation}\label{bg}
		\lim_{N\to\infty}\E_{\rho}^N\cro{\pa{\int_0^t ds \frac{1}{\sqrt{N}}\sum_{x\in \Z}G(x/N)\tau_x V_\psi(\eta(s))}^2}=0,
	\end{equation}
	where 
	\[V_\psi(\eta)=\psi(\eta)-\widetilde{\psi}(\rho)-\widetilde{\psi}'(\rho)(\eta_0-\rho),\]
	and $\tau_x \psi(\eta)=\psi(\eta_{\cdot-x})$ denotes the translation of $\psi$ by $x$.
\end{proposition}

By Proposition \ref{thm:bg} and Lemma \ref{lem:qv}, for any limit $Y_\cdot$ of $\mathcal{Y}^N_\cdot$ and any compactly supported test functions $G\in{ \mc S}$,
\[Y_t (G) - Y_0 (G) - D(\rho) \int_0^t  Y_s (\partial_u^2 G) ds\]
is a martingale with quadratic variation
\[t \sigma(\rho) ||\partial_u G||_{L^2 (\R)},\]
which proves Theorem \ref{thm:symFEP}. We now prove Proposition \ref{thm:bg}.

\subsection{Boltzmann-Gibbs Principle}
\label{sec:bg}
In this section, we prove Proposition \ref{thm:bg}. Since it is the main case of interest, let us assume $\psi$ depends only on the values of $\eta_0$ and $\eta_{\pm 1}$. The proof is divided into several steps, and it can be straightforwardly adapted when $\psi$ depends on the value of $\eta$ in a finite box $B_\psi$.

\noindent {\bf Step 1}.  Fix $\ell > 0$, which will goes to infinity after $N \rightarrow \infty$. For $i\in \Z$, fix points $x_i <y_i$ in $\Z$ such that $x_0=0$, and for $i\in \Z$
\[y_i-x_i = \ell-1, \qquad y_i + 3 = x_{i+1}.\]
Let $A_i = \{x_i, \dots, y_i\}$. Then, the  length of the interval $A_i$ is $\ell$ for each $i$.  Denote
\[A_i^o = \{x_i+1,\dots,y_i-1 \}, \qquad B_i = \{y_i,\dots,x_{i+1}\}.\]
Fix some point $z_i \in A_i^o$. Then, we rewrite the sum inside the expectation in \eqref{bg} as
\begin{multline}\label{bg1}
	\sum_{x\in \Z} G(x/N)\tau_x V_\psi =  \sum_{i \in \Z} \sum_{x \in  B_i} G(x/N)\tau_x V_\psi + \sum_{i \in \Z} \sum_{x \in A_i^o} \big(G(x/N) - G(z_i/N)\big) \tau_x V_\psi\\
	+ \sum_{i \in \Z} G(z_i/N) \sum_{x \in A_i^o} \tau_x V_\psi.
\end{multline}
For the first term above, by Cauchy-Schwarz inequality and the invariance of the measure $\pi_\rho$, 
\begin{multline}
	\E_{\rho}^N \cro{\pa{\int_0^t ds \frac{1}{\sqrt{N}} \sum_{i \in \Z} \sum_{x\in {B}_i}G(x/N)\tau_x V_\psi(\eta_s)}^2} \leq t^2 E_{\pi_\rho}\cro{\pa{  \frac{1}{\sqrt{N}}\sum_{i\in \Z} \sum_{x \in {B}_i} G(x/N)\tau_x V_\psi(\eta)}^2}\\
	=   \frac{t^2}{N}\sum_{i \neq i^\prime}  {\rm Cov}_{\pi_\rho}\cro{  \sum_{x \in {B}_i} G(x/N)\tau_x V_\psi(\eta),\sum_{x \in {B}_{i^\prime}} G(x/N)\tau_x V_\psi(\eta)} 
	\\+ \frac{t^2}{N}\sum_{i \in \Z} E_{\pi_\rho}\cro{\pa{  \sum_{x \in {B}_i} G(x/N)\tau_x V_\psi(\eta)}^2}.
\end{multline}
Since $|B_i|=4$,  it is easy to see that there exists a constant $C = C(G,\psi)$ such that the second term on the right hand side is bounded by $C t^2 / \ell$, which vanishes in the limit as $\ell \rightarrow \infty$. Since the distance between the two boxes $B_i$ and $B_{i^\prime}$ is of order $|i-i^{\prime}| \ell$, by the exponential decay of correlations of  the measure $\pi_\rho$ (\emph{cf.}\,\cite[Corollary 6.6]{blondel2020hydrodynamic}), there exists a constant $C$ independent of $N$ such that the first term in the last inequality is bounded by 
\[ \frac{t^2}{N}\sum_{i \neq i^\prime} E_{\pi_\rho}\cro{ \Big|  \sum_{x \in {B}_i} G(x/N)\tau_x V_\psi(\eta) \Big| } E_{\pi_\rho}\cro{ \Big|  \sum_{x \in {B}_{i^\prime}} G(x/N)\tau_x V_\psi(\eta) \Big| } e^{-C \ell |i-i^\prime|}.\]
We may bound the above term by 
\[  \frac{C t^2}{N} \sum_{i,i^\prime = -C N/\ell}^{CN/\ell}  e^{-C \ell |i-i^\prime|} \leq  \frac{Ct^2}{\ell}\]
for some constant $C = C(G,\psi)$. 

For the contribution of the second term in \eqref{bg1}, by smoothness of the function $G$, there exists a constant $C = C(G,\psi)$ such that 
\[\Big|\frac{1}{\sqrt{N}} \sum_{i \in \Z} \sum_{x \in A_i^o} \big(G(x/N) - G(z_i/N)\big) \tau_x V_\psi \Big| \leq \frac{C\ell}{\sqrt{N}},\] 
which vanishes as $N \rightarrow \infty$.

\noindent{\bf Step 2.} It remains to deal with the contribution of third term in \eqref{bg1}. For any finite set $F \subset \Z$, denote by $\eta^F$ the empirical density of particles in the set $F$ in configuration $\eta$,
\[\eta^F = \frac{1}{|F|} \sum_{x\in F} \eta_x.\]
For $i \in \Z$, define 
\[\widetilde{V}_{\psi,i} (\eta)= E_{\pi_\rho} \Big[  \sum_{x \in A_i^o} \tau_x V_\psi \big| \eta^{A_i}, \eta_{x_i-1}, \eta_{y_{i}+1} \Big].\]
Then, by Cauchy-Schwarz inequality,
\begin{multline}\label{eqn:cs1}
	\E_{\rho}^N \left[ \Big(  \int_0^t ds \frac{1}{\sqrt{N}} \sum_{i \in \Z}  G(z_i/N)  \sum_{x \in A_i^o} \tau_x V_\psi (\eta(s))\Big)^2 \right] \\
	\leq 2 \E_{\rho}^N \left[ \Big(  \int_0^t ds \frac{1}{\sqrt{N}} \sum_{i \in \Z} G(z_i/N) \Big\{ \sum_{x \in A_i^o} \tau_x V_\psi (\eta(s)) - \widetilde{V}_{\psi,i} (\eta(s)) \Big\}\Big)^2 \right]\\
	+ 2 \E_{\rho}^N \left[ \Big(  \int_0^t ds \frac{1}{\sqrt{N}} \sum_{i \in \Z} G(z_i/N) \widetilde{V}_{\psi,i} (\eta(s))\Big)^2 \right].
\end{multline}
We first prove the second term above converges to zero as $N \rightarrow \infty, \ell \rightarrow \infty$. Shorten 
\[A=A_0=\{0,\dots,\ell-1\},\]  
and $ \widetilde{V}_{\psi} =  \widetilde{V}_{\psi,0}$. By Cauchy-Schwarz inequality and the exponential decay of correlations of  the measure $\pi_\rho$ (\emph{cf.}\,\cite[Corollary 6.6]{blondel2020hydrodynamic}), there exists a constant $C=C(G)$ such that
\begin{multline}\label{cs2}
	\E_{\rho}^N \left[ \Big(  \int_0^t ds \frac{1}{\sqrt{N}} \sum_{i \in \Z} G(z_i/N) \widetilde{V}_{\psi,i} (\eta(s))\Big)^2 \right] \\
	\leq C t^2 \ell E_{\pi_\rho} \Big[ \Big( E_{\pi_\rho} \Big[ \frac{1}{\ell}\sum_{x \in A^o} \tau_x \psi   \big| \eta^{A}, \eta_{x_0-1}, \eta_{y_{0}+1} \Big] - \widetilde{\psi}(\eta^A)\Big)^2 \Big]
	\\ + C t^2 \ell E_{\pi_\rho} \Big[ \big( \widetilde{\psi}(\eta^A)-\widetilde{\psi}(\rho)-\widetilde{\psi}^\prime(\rho)(\eta^A-\rho)\big)^2 \Big] + \mathcal{O} (\ell^{-1}).
\end{multline}
Fix $\varepsilon_0 > 0$ such that $\rho- \varepsilon_0 > 1/2$ and that $\rho+\varepsilon_0 < 1$. We bound the first term above by
\[C t^2 \ell \Big\{ E_{\pi_\rho} \Big[ \Big( E_{\pi_\rho} \Big[ \frac{1}{\ell}\sum_{x \in A^o} \tau_x \psi   \big| \eta^{A}, \eta_{x_0-1}, \eta_{y_{0}+1} \Big] - \widetilde{\psi}(\eta^A)\Big)^2 \mathbbm{1}_{|\eta^A-\rho| <\varepsilon_0} \Big] 
+  P_{\pi_\rho} \Big( |\eta^A-\rho| \geq \varepsilon_0 \Big)\Big\}.\] 
According to Corollary \ref{cor:Markov}, under $\pi_\rho$, $\{\eta_x\}_{x\geq 0}$ is an ergodic Markov chain with finite state space $\{0,1\}$, and with transition probability 
	\[p(0,1) = 1, \quad p(1,1) = 1 - p (1,0) = \frac{2\rho-1}{\rho}.\]
	By \cite{katz1960exponential}, for any $\varepsilon > 0$, there exists a finite constant $C$ such that
	\[P_{\pi_\rho} \Big( |\eta^A-\rho| \geq \varepsilon \Big) \leq C \varepsilon^{-2} e^{-C \ell \varepsilon^2}.\]
	Together with Proposition \ref{pro:equiv1}, the first term in \eqref{cs2} is bounded by
	\[C t^2 \Big[ (\log \ell)^4 / \ell + \ell e^{-C \ell}\Big].\]
	By Taylor's expansion and the above large deviation estimates, the second term in \eqref{cs2} is bounded by
	\begin{multline*}
		C t^2 \ell E_{\pi_\rho} \big[ (\eta^A-\rho)^4 \big] \leq C t^2 \ell \Big[ E_{\pi_\rho} \big[ (\eta^A-\rho)^4 \mathbbm{1}_{|\eta^A-\rho| < \ell^{-1/3}} \big] + P_{\pi_\rho} \Big( |\eta^A-\rho| > \ell^{-1/3} \Big)  \Big]\\
		\leq Ct^2 \Big[ \ell^{-1/3} + \ell^{5/3} e^{-C \ell^{1/3}} \Big].
	\end{multline*}

\noindent{\bf Step 3.} To conclude the proof, it remains to show
\begin{equation}\label{kv0}
	\lim_{N\to\infty} \E_{\rho}^N \left[ \Big(  \int_0^t ds \frac{1}{\sqrt{N}} \sum_{i \in \Z} G(z_i/N) \Big\{ \sum_{x \in A_i^o} \tau_x V_\psi (\eta(s)) - \widetilde{V}_{\psi,i} (\eta(s)) \Big\}\Big)^2 \right] = 0.
\end{equation}
By Kipnis-Varadhan's inequality (see e.g. \cite[Proposition A1.6.1, p333]{klscaling}), the above term is bounded by
\begin{equation}\label{kv1}
	20 t \sup_{f \in L^2 (\pi_\rho)} \Big\{ \frac{1}{\sqrt{N}} \sum_{i \in \Z} G(z_i/N) E_{\pi_\rho} \Big[  \Big\{ \sum_{x \in A_i^o} \tau_x V_\psi - \widetilde{V}_{\psi,i} \Big\} f \Big] -  E_{\pi_\rho} \big[f(-\genex) f\big] \Big\}.
\end{equation}
To make notations short, let $f_i = \tau_{x_0-x_i} f$. Define $\bar{A}=A\cup\{x_0-1,y_0+1\}$.  For $\sigma \in \{0,1\}^{\bar{A}}$, denote
\[\bar{f} (\sigma) = E_{\pi_\rho} [f | \eta_{|\bar{A}} = \sigma].\]
Since $\pi_\rho$  is translation invariant, and $\sum_{x \in A^o} \tau_x V_\psi - \widetilde{V}_{\psi}$ depends only  on the values of $(\eta_x, x \in \bar{A})$, 
\begin{equation}
	\label{eq:Epirhofibar}
	E_{\pi_\rho} \Big[  \Big\{ \sum_{x \in A_i^o} \tau_x V_\psi - \widetilde{V}_{\psi,i} \Big\} f \Big] =  E_{\pi_\rho} \Big[  \Big\{ \sum_{x \in A^o} \tau_x V_\psi - \widetilde{V}_{\psi} \Big\} f_i \Big] = E_{\pi_\rho} \Big[  \Big\{ \sum_{x \in A^o} \tau_x V_\psi - \widetilde{V}_{\psi} \Big\} \bar{f_i} \Big].
\end{equation}
We now project on hyperplanes with fixed number of particles in $A=\{x_0\dots, y_0\}$. However, $\pi_\rho$ is not constant on such a set because the probability to see a local configuration depends on its boundary values. For this reason, aside from conditionning to the number of particles in $A$, we also condition to fixed boundary values; For any triplet $\bar k=(k,a,b)\in \{0,\dots,\ell\}\times \{0,1\}^2$, define 
\[\Sigma_{A,\bar k}=\left\{\eta\in \{0,1\}^\Z, \; \sum_{z\in A} \eta_z=k,\; \eta_{x_0-1}=a, \; \eta_{y_0+1}=b\right\},\]
the set of configurations on $\bar{A}$ with $k$ particles in $A$ and with boundary states $(a,b)$.
We can now rewrite the last term in the right-hand side of \eqref{eq:Epirhofibar} as
\[\sum_{\bar k} m_{\bar k} E_{A,\bar k}\Big[  \Big\{ \sum_{x \in A^o} \tau_x V_\psi - \widetilde{V}_{\psi} \Big\} \bar{f_i} \Big],\]
where $m_{\bar k} = \pi_\rho (\Sigma_{A, \bar k})$, and for any $g$ defined on $ \{0,1\}^{\bar A}$,
\[E_{A,\bar k}[g] = E_{\pi_\rho} \big[\,g(\eta_{\mid \bar A})\, |\, \eta\in \Sigma_{A,\bar k}\big].\]
For any $\bar k=(k,a,b)\in \{0,\dots,\ell\}\times \{0,1\}^2$ and  any function $g:\{0,1\}^{A} \rightarrow \R$, define $\mathcal{L}_{A,\bar k}$  as
\[\mathscr{L}_{A, \bar k}g (\eta) = \sum_{z =x_0}^{y_0} c_{x,x+1} (\eta) (g(\eta^{x,x+1})-g(\eta)) \]
with the convention that $\eta_{x_0-1} = a, \eta_{y_0+1} = b$. Note tha the generator does not in fact depend on the number $k$ of particles in $A$, just on the boundary states $a$ and $b$.  Recall that $\pi_\rho$ only charges ergodic configurations, we claim that $\sum_{x \in A^o} \tau_x V_\psi - \widetilde{V}_{\psi}$  is in the range of $\mathscr{L}_{A,\bar k}$, i.e. that there exists a function $g$ on $\Sigma_{A,\bar k}$ such that \[\sum_{x \in A^o} \tau_x V_\psi - \widetilde{V}_{\psi}=\mathscr{L}_{A, \bar k}g.\] 
The claim follows directly from the fact that $E_{A,\bar k}\big[\sum_{x \in A^o} \tau_x V_\psi - \widetilde{V}_{\psi} \big] = 0$ and from the following observations:
\begin{enumerate}[(i)]
	\item the codimension of $\mathscr{L}_{A, \bar k}$'s range is one in the space of functions on $\Sigma_{A, \bar k}$. Indeed, if $\mathscr{L}_{A, \bar k} g  = 0$, then $E_{A,\bar k}\big[ g\mathscr{L}_{A, \bar k}g \big] = 0,$ implying that $g$ is constant on $\Sigma_{A, \bar k}$ by ergodicity. In particular, the dimension of the kernel of $\mathscr{L}_{A,\bar k}$ is one, which proves the claim.
	\item The set of mean-$0$ functions is also of codimension $1$, and it contains the range of $\mathscr{L}_{A, \bar k}$. In particular, the range of $\mathscr{L}_{A, \bar k}$ is the set of mean-$0$ functions.
\end{enumerate}
Then, by Cauchy-Schwarz inequality, for any $\gamma > 0$, 
\begin{multline*}
	E_{A,\bar k} \Big[  \Big\{ \sum_{x \in A^o} \tau_x V_\psi - \widetilde{V}_{\psi} \Big\} \bar{f_i} \Big] \leq \gamma E_{A,\bar k} \Big[  \Big\{ \sum_{x \in A^o} \tau_x V_\psi - \widetilde{V}_{\psi} \Big\} (-\mathscr{L}_{A,\bar k})^{-1} \Big\{ \sum_{x \in A^o} \tau_x V_\psi - \widetilde{V}_{\psi} \Big\}\Big] \\ + \gamma^{-1} E_{A,\bar k} \Big[  \bar{f_i} (-\mathscr{L}_{A,\bar k}) \bar{f_i}\Big].
\end{multline*}
Therefore, we may bound the first term in \eqref{kv1} by
\begin{multline}\label{kv2}
	\frac{1}{\sqrt{N}} \sum_{i \in \Z} G(z_i/N) E_{\pi_\rho} \Big[  \Big\{ \sum_{x \in A_i^o} \tau_x V_\psi - \widetilde{V}_{\psi,i} \Big\} f \Big] \\
	\leq 	\frac{1}{\sqrt{N}} \sum_{i \in \Z} \gamma G(z_i/N) \sum_{\bar k}  m_{\bar k}  E_{A,\bar k} \Big[  \Big\{ \sum_{x \in A^o} \tau_x V_\psi - \widetilde{V}_{\psi} \Big\} (-\mathscr{L}_{A,\bar k})^{-1} \Big\{ \sum_{x \in A^o} \tau_x V_\psi - \widetilde{V}_{\psi} \Big\}\Big] \\
	+ 	\frac{1}{\sqrt{N}} \sum_{i \in \Z} \gamma^{-1}  G(z_i/N) \sum_{\bar k}  m_{\bar k}  E_{A,\bar k} \Big[  \bar{f_i} (-\mathscr{L}_{A,\bar k}) \bar{f_i}\Big].
\end{multline}

It is also easy to see that there exists a constant $C_\ell = C_\ell (\psi)$  such that
\[E_{A,\bar k} \Big[  \Big\{ \sum_{x \in A^o} \tau_x V_\psi - \widetilde{V}_{\psi} \Big\} (-\mathscr{L}_{A,\bar k})^{-1} \Big\{ \sum_{x \in A^o} \tau_x V_\psi - \widetilde{V}_{\psi} \Big\}\Big] \leq C_\ell.\]
By convexity and translation invariance of the Dirichlet form,
\[ N^2 \sum_{i} \sum_{\bar{k}}  m_{\bar k}  E_{A,\bar k} \Big[  \bar{f_i} (-\mathscr{L}_{A,\bar k}) \bar{f_i}\Big] \leq E_{\pi_\rho} \big[f (-\genex) f\big].\]
Taking $\gamma =   G(z_i/N) / (N^{5/2} \ell)$, we may bound \eqref{kv1} by $C_\ell t/ N^2$. This proves \eqref{kv0} and concludes the proof of the theorem.

\subsection{Equivalence of ensembles.} 
In this subsection, we improve on the equivalence of ensembles estimate given in \cite{blondel2020hydrodynamic}. Throughout this section, for $\ell > 0$ and $x\in \Z$  we define 
\[B_\ell(x)=\{x-\ell,\dots,x+\ell\} \qquad\mbox{ and }\qquad B_\ell:=B_\ell(0)=\{-\ell,\dots,\ell\}.\]
For any $\delta$, introduce 
\[E_\ell (\delta)= \{  (\ell + 1)(1+\delta) , \ldots, (2 \ell+1)(1-\delta)\} \qquad E_\ell:=E_\ell(0)= \{  \ell + 1 , \ldots, 2 \ell+1\},\]
which are the possible numbers of particles in $B_\ell$ after cropping densities $\delta$ close to $1/2$ and $1$.

Fix two boundary conditions $a:=(a_1,a_2)\in \{0,1\}^2$, and  $j \in E_\ell$,  and let
\begin{equation*}
	\mathcal{E}_{\ell,j}^a= \Big\{\eta \in \mathcal{E}_{B_\ell}, \;a_1+\eta_{-\ell}\geq 1, \; a_2+\eta_\ell \geq 1\; \text{and} \; \sum_{x\in B_\ell} \eta_x = j\Big\},
\end{equation*}
which is the set of configurations on $B_\ell$ with j particles, and which are ergodic when supplemented with boundary conditions $\eta_{-\ell-1}=a_1$, $\eta_{\ell+1}=a_2$.
Let $\nota{\pi_{\ell,j}^a}$ be the uniform measure on $\mathcal{E}_{\ell,j}^a$, and denote by $\nota{\E_{\ell,j}^a}$ the corresponding expectation.  Denote $\nota{\rho_\ell (j) = j / (2\ell +1)}$. 

\begin{proposition}
\label{pro:equiv}
	Fix an integer $k\geq 0$ and a local ergodic configuration $\sigma \in \mathcal{E}_{B_k}$. Then,  for any $\delta > 0$, there exists a constant $C>0$ such that for any $j \in E_\ell (\delta)$ and $a\in \{0,1\}^2$ 
	\begin{equation}
	\max_{x \in B_{\ell - (\log \ell)^2}}   |\pi_{\ell,j}^a (\eta_{|B_k(x)} = \sigma) - \pi_{\rho_\ell (j)} (\eta_{|B_k} = \sigma)| \leq \frac{C (\log \ell)^2}{\ell}.
	\end{equation}
\end{proposition}
Note that this estimate is much sharper than the one $O(\ell^{-1/4})$ obtained in \cite{blondel2020hydrodynamic}. Further note that the sharp bound for the SSEP would be of order $1/\ell$, and since the FEP's stationary states are locally correlated, it is natural to have corrections of order $\log(\ell)$ w.r.t. the SSEP.  Low densities below $1/2+\delta/2$ are excluded because as  $\rho\to1/2$, long range correlations appear in $\pi_\rho$, whereas high densities larger than $1-\delta$ are excluded for technical reasons.
As a direct consequence of the above proposition, we have the following version of equivalence of ensembles. Since the proof is straightforward, we do not detail it here.

\begin{proposition}[Equivalence of ensembles]\label{pro:equiv1} Let $f: \mathcal{E}_\Z \rightarrow \R$ be a local function, whose support is contained in $B_{\ell_0}$ for some $\ell_0 >0$. Then, for any $\delta > 0$, there exists a constant $C>0$ such that for any $a\in\{0,1\}^2$, and any $j \in E_\ell (\delta)$,
	\[\Big| \frac{1}{\ell} \sum_{x \in B_{\ell - \ell_0}} \E_{\ell,j}^a [\tau_x f] -  \E_{\pi_{\rho_\ell (j)}} [f]\Big| \leq \frac{C (\log \ell)^2}{\ell}.\]
\end{proposition}

In the remainder of this subsection, we prove  Proposition \ref{pro:equiv}. For integers $\lfloor\ell/2\rfloor \leq j \leq \ell$, let $N_{\ell,j}$ be the number of ergodic configurations in $\nota{\Lambda_\ell:=\{1,\dots,\ell\}}$ with $j$ particles. It is easy to see that 
\[N_{\ell,j} ={ j+1 \choose \ell - j},\]
since to build an ergodic configuration, the $\ell-j$ empty sites need to be placed at one of the extremities, or at one of the $j-1$ places in-between particles. We start by stating and proving two technical lemmas.

\begin{lemma}\label{lem:1}
	Fix $\ell_1 < \ell_2$ such that $\ell_1 + \ell_2 = \ell$. Then
	\begin{equation}\label{eqn:1}
		\sum_{j_1+j_2 = j} N_{\ell_1,j_1} \times N_{\ell_2,j_2} = N_{\ell,j} + \sum_{j_1+j_2 = j-2} N_{\ell_1-2,j_1} \times  N_{\ell_2-2,j_2}.
	\end{equation}
	By induction, if $\ell_1$ is odd, then this formula yields
	\[	\sum_{j_1+j_2 = j} N_{\ell_1,j_1} \times N_{\ell_2,j_2}  = \sum_{m=0}^{(\ell_1-3)/2} N_{\ell-4m,j-2m} + \sum_{j_1+j_2 = j - \ell_1+1} N_{1,j_1} \times N_{\ell_2-\ell_1+1,j_2},\]
	and if $\ell_1$ is even, then
	\[	\sum_{j_1+j_2 = j} N_{\ell_1,j_1} \times N_{\ell_2,j_2}  = \sum_{m=0}^{(\ell_1-4)/2} N_{\ell-4m,j-2m} + \sum_{j_1+j_2 = j - \ell_1+2} N_{2,j_1} \times N_{\ell_2-\ell_1+2,j_2}.\]
\end{lemma}

\begin{proof}
	Note that the left hand side of \eqref{eqn:1} is the number of configurations  $\eta \in \{0,1\}^{\Lambda_\ell}$ with $j$ particles, which are ergodic on $\{1,\dots \ell_1\}$ and on $\{\ell_1+1,\dots \ell\}$. Any such configuration is either ergodic, or such that $\eta_{\ell_1}=\eta_{\ell_1+1}=0$,  $\eta_{\ell_1-1}=\eta_{\ell_1+2}=1$, and is ergodic both on $\{1,\dots,\ell_1-2\}$ and on $\{\ell_1+3,\dots,\ell\}$, which yields \eqref{eqn:1}.
	\end{proof}

\begin{lemma}\label{lem:2}
	Fix $\rho \in [1/2,1]$. Then, there exists a constant $C=C(\rho)$ such that, for any integer $\ell > 0$ large enough,
	\begin{equation}\label{eqn:3}
		0 \leq 	\sum_{m=1}^{\ell - 1} \Big\{ \prod_{n=1}^{m} \frac{ 1-\rho+\tfrac{n}{\ell}}{ \rho+\tfrac{n}{\ell}} - \Big(\frac{1-\rho}{\rho}\Big)^m \Big\} \leq \frac{C (\log \ell)^2}{\ell}.
	\end{equation}
\end{lemma}

\begin{proof}
	It is easy to see the term in \eqref{eqn:3} is non-negative. The result is trivial if $\rho = 1/2$. Now suppose $\rho \in (1/2,1)$.   Denote
	\[a_m = \prod_{n=1}^m \frac{1+\tfrac{n}{(1-\rho)\ell}}{1+\tfrac{n}{\rho\ell}} - 1.\]
	Fix $M=M(\rho)$ such that $M \log \tfrac{1+\rho}{2-\rho} > 1$. If $m \leq M \log \ell$, then
	\[a_m \leq \Big(\frac{1+\tfrac{M\log \ell}{(1-\rho)\ell}}{1+\tfrac{M \log \ell}{ \rho\ell}}  \Big)^{M \log \ell} - 1 \leq e^{C(\rho) (\log \ell)^2 / \ell} - 1 \leq \frac{C(\rho) (\log \ell)^2}{\ell}.\]
	Therefore,
	\[\sum_{m=1}^{\log \ell} \Big\{ \prod_{n=1}^{m} \frac{ 1-\rho+\tfrac{n}{\ell}}{ \rho+\tfrac{n}{\ell}} - \Big(\frac{1-\rho}{\rho}\Big)^m \Big\} =  \sum_{m=1}^{\log \ell}  a_m \Big(\frac{1-\rho}{\rho}\Big)^m \leq \frac{C(\rho) (\log \ell)^2}{\ell}.\]
	Observe that
	\[		a_{m+1} = \frac{1+\tfrac{m+1}{(1-\rho)\ell}}{1+\tfrac{m+1}{\rho\ell}}  (a_m+1) - 1 = \frac{1+\tfrac{m+1}{(1-\rho)\ell}}{1+\tfrac{m+1}{\rho\ell}} a_m + (m+1) \frac{\tfrac{1}{(1-\rho)\ell} - \tfrac{1}{\rho \ell}}{1+\tfrac{m+1}{\rho \ell}}.\]
	Since $m +1\leq \ell$,
	\[a_{m+1} \leq \frac{\rho (2-\rho)}{(1-\rho)(1+\rho)} a_m + \frac{C(\rho) m }{\ell}.\]	
	In particular,
	\[a_m \leq C(\rho) \Big[ \Big( \frac{\rho (2-\rho)}{(1-\rho)(1+\rho)}\Big)^m + \frac{m}{\ell} \Big].\]
	Since $2-\rho < 1+\rho$,
	\begin{multline*}
		\sum_{m=M \log \ell}^{\ell-1} \Big\{ \prod_{n=1}^{m} \frac{ 1-\rho+\tfrac{n}{\ell}}{ \rho+\tfrac{n}{\ell}} - \Big(\frac{1-\rho}{\rho}\Big)^m \Big\} = \sum_{m= M \log \ell}^{\ell-1}   a_m \Big(\frac{1-\rho}{\rho}\Big)^m\\
		\leq C(\rho) \sum_{m=M \log \ell}^{\ell-1}  \Big(\frac{1-\rho}{\rho}\Big)^m  \Big[ \Big( \frac{\rho (2-\rho)}{(1-\rho)(1+\rho)}\Big)^m + \frac{m}{\ell} \Big] \leq C(\rho) \Big( \big(\tfrac{2-\rho}{1+\rho}\big)^{M \log \ell} + \ell^{-1}\Big)
		\leq \frac{C(\rho)}{\ell}.
	\end{multline*}
	The case $\rho = 1$ is easier and could be proved in the same way. This concludes the proof.
\end{proof}

We are now ready to prove Proposition \ref{pro:equiv}.

\begin{proof}[Proof of Proposition \ref{pro:equiv}] Let $j_0 = \sum_{y \in B_k} \eta_y$ be the number of particles in $\eta$. Fix a number $j$ of particles, and boundary conditions $a\in\{0,1\}^2$.  Then,
	\[\pi^a_{\ell,j} (\eta_{|B_k(x)} = \sigma)  = \frac{\sum_{j_1+j_2 = j - j_0} {j_1+a_1+ \eta_{-k} - 1 \choose x-k+\ell - j_1} {j_2+a_2 + \eta_{k} - 1 \choose \ell -x-k- j_2} }{{j+a_1+a_2-1 \choose 2\ell+1-j}},\]
	and
	\[\pi_{\rho_\ell (j)} (\eta_{|B_k} = \sigma) = (1-\rho_\ell(j)) \Big(\frac{1-\rho_\ell(j)}{\rho_\ell(j)}\Big)^{2k - j_0} \Big(\frac{2\rho_\ell(j)-1}{\rho_\ell(j)}\Big)^{2j_0 - 2k - \eta_{-k} - \eta_{k}}.\]

	We only consider the case where $\ell + x -k +a_1 + \eta_{-k}-2$ is odd, and the other case can be treated in the same way. Without loss of generality, we assume $x < 0$.  By Lemma \ref{lem:1},
	\begin{multline*}
		\sum_{j_1+j_2= j - j_0} {j_1+a_1 + \eta_{-k} - 1 \choose x-k+\ell - j_1} {j_2+a_2 + \eta_{k} - 1 \choose \ell -x-k- j_2}\\
		 {= \sum_{j_1+j_2= j - j_0} N_{\ell+x-k+a_1+\eta_{-k}-2,j_1+a_1+\eta_{-k}-2} N_{\ell-x-k+a_2+\eta_k-2,j_2+a_2+\eta_k-2}}\\
		 {= \sum_{m=0}^{(x-k+\ell+a_1+\eta_{-k}-5)/2} N_{2\ell-2k+a_2+a_1
		 		+\eta_{k}+\eta_{-k}-4-4m,j-j_0+a_2+a_1
		 		+\eta_{k}+\eta_{-k}-4-2m} }\\
	 		{+ \sum_{j_1+j_2=j-j_0-x+k-\ell +a_2+\eta_{k} -1} N_{1,j_1} N_{-2x+a_2-a_1+\eta_k-\eta_{-k}+1,j_2}}\\
		= \sum_{m=0}^{(x-k+\ell+a_1+\eta_{-k}-5)/2}  {j-j_0+a_2+a_1
			+\eta_{k}+\eta_{-k}-3-2m \choose 2\ell - 2k -j+j_0-2m} \\+ {j-j_0-x+k-\ell +a_2+\eta_{k} \choose \ell - x-k-j+j_0+2-a_1-\eta_{-k}} + {j-j_0-x+k-\ell +a_2+\eta_{k} -1 \choose \ell - x-k-j+j_0+3-a_1-\eta_{-k}}.
	\end{multline*}
	Therefore,
	\[	\pi_{\ell,j}^a (\sigma_{|B_k(x)} = \eta)  = \sum_{m=0}^{(x-k+\ell+a_1+\eta_{-k}-5)/2}  \frac{{j-j_0+a_2+a_1
			+\eta_{k}+\eta_{-k}-3-2m \choose 2\ell - 2k -j+j_0-2m}}{{j+a_1 +a_2-1 \choose 2\ell+1-j}}  + \mathcal{F}_{\ell,j} (x,k,\eta),  \]
	where
	\begin{equation}\label{eqn:error1}
		\mathcal{F}_{\ell,j} (x,k,\eta) = \frac{ {j-j_0-x+k-\ell +a_2+\eta_{k} \choose \ell - x-k-j+j_0+2-a_1-\eta_{-k}} + {j-j_0-x+k-\ell +a_2+\eta_{k} -1 \choose \ell - x-k-j+j_0+3-a_1-\eta_{-k}}}{{j+a_1 +a_2-1 \choose 2\ell+1-j}}.
	\end{equation}
	In Lemma \ref{lem:3} below, we  prove that $| \mathcal{F}_{\ell,j} (x,k,\eta)| \leq C\ell^{-1}$. Observe that
	\begin{multline*}
		\frac{{j-j_0+a_2+a_1
				+\eta_{k}+\eta_{-k}-3-2m \choose 2\ell - 2k -j+j_0-2m}}{{j+a_1 +a_2-1 \choose 2\ell+1-j}} = \frac{(j-j_0+a_2+a_1
			+\eta_{k}+\eta_{-k}-3-2m)!}{(j+a_1 +a_2-1)!} \\\times \frac{(2\ell+1-j)!}{(2\ell - 2k -j+j_0-2m)!} \times  \frac{(2j-2\ell+a_1+a_2-2)!}{(2j-2\ell+2k-2j_0-3+a_2+a_1+\eta_{k}+\eta_{-k})!}
	\end{multline*}
	By Lemma \ref{lem:2},
	\begin{multline*}
		\sum_{m=0}^{(x-k+\ell+a_1+\eta_{-k}-5)/2} \Big| \frac{{j-j_0+a_2+a_1
				+\eta_{k}+\eta_{-k}-3-2m \choose 2\ell - 2k -j+j_0-2m}}{{j+a_1 +a_2-1 \choose 2\ell+1-j}} \\
		-  \Big(\frac{1}{\rho_\ell(j)}\Big)^{j_0+2-\eta_{k}-\eta_{-k}+2m} \Big(1-\rho_\ell(j)\Big)^{1+2k-j_0+2m} \Big( 2 \rho_\ell(j)-1 \Big)^{2j_0-2k+1-\eta_{k} - \eta_{-k}}  \Big| \leq \frac{C (\log \ell)^2}{\ell}.
	\end{multline*}
	Note also that
	\begin{multline*}
		\sum_{m=0}^{\infty} \Big(\frac{1}{\rho_\ell(j)}\Big)^{j_0+2-\eta_{k}-\eta_{-k}+2m} \Big(1-\rho_\ell(j)\Big)^{1+2k-j_0+2m} \Big( 2 \rho_\ell(j)-1 \Big)^{2j_0-2k+1-\eta_{k} - \eta_{-k}} \\
		= \Big(\frac{1}{\rho_\ell(j)}\Big)^{j_0-\eta_{k}-\eta_{-k}} \Big(1-\rho_\ell(j)\Big)^{1+2k-j_0} \Big( 2 \rho_\ell(j)-1 \Big)^{2j_0-2k-\eta_{k} - \eta_{-k}} = \pi_{\rho_\ell (j)} (\eta_{|B_k} = \sigma) .
	\end{multline*}
	Therefore, 
	\begin{multline*}
		\Big| 	\pi_{\ell,j}^a (\eta_{|B_k(x)} = \sigma)  - \pi_{\rho_\ell (j)} (\eta_{|B_k} = \sigma)\Big| \leq  \frac{C (\log \ell)^2}{\ell}  \\
		+ \sum_{m=\ell/4}^{\infty} \Big(\frac{1}{\rho_\ell(j)}\Big)^{j_0+2-\eta_{k}-\eta_{-k}+2m} \Big(1-\rho_\ell(j)\Big)^{1+2k-j_0+2m} \Big( 2 \rho_\ell(j)-1 \Big)^{2j_0-2k+1-\eta_{k} - \eta_{-k}}\\
		\leq  \frac{C (\log \ell)^2}{\ell}.
	\end{multline*}
	This concludes the proof.
\end{proof}

\begin{lemma}\label{lem:3}
	The error term defined in \eqref{eqn:error1} satisfies
	\[| \mathcal{F}_{\ell,j} (x,k,\eta)| \leq C\ell^{-1}.\]
\end{lemma}

\begin{proof}
	We only prove 
	\[\frac{ {j-j_0-x+k-\ell +a_2+\eta_{k} \choose \ell - x-k-j+j_0+2-a_1-\eta_{-k}} }{{j+a_1 +a_2-1 \choose 2\ell+1-j}} \leq C\ell^{-1},\]
	since the remaining term could be treated in the same way. Developing the above factorial, we rewrite the left term as
	\begin{multline*}
		\frac{(j-j_0-x+k-\ell +a_2+\eta_{k})!}{(j+a_1 +a_2-1)!}  \times  \frac{(2\ell+1-j)!}{( \ell - x-k-j+j_0+2-a_1-\eta_{-k})!}\\ \times  \frac{(2j-2\ell+a_1+a_2-2)!}{(2j-2\ell-2+2k-2j_0+\eta_{-k}+\eta_{k}+a_1+a_2)!}.
	\end{multline*}
	There exists some constant $C$ such that the above term is bounded by
	\[C \Big(\frac{1-\rho_\ell(j)}{\rho_\ell(j)}\Big)^{\ell+x}.\]
	Since $-\ell+(\log \ell)^2 \leq x \leq 0$ and $\rho_\ell (j) > (1+\delta)/2 $, it is easy to see the last line is bounded by $C \ell^{-1}$. This concludes the proof.
\end{proof}

\appendix 
\section{Stationary distributions}\label{sec:stationaryDist}
\subsection{Proof of \eqref{eq:mappingdistributions}}
\label{sec:statPi}
In this section, we give a more straightforward interpretation of the stationary measure $\pi_\rho$, and prove \eqref{eq:mappingdistributions}. First note that \eqref{eq:mappingdistributions} yields the following construction for $\pi_\rho$: we first build the central cluster, by sampling $\omega_0$ according to \eqref{eq:centralcluster}, and then randomly translating it by choosing $X_0$ uniformly in $\{-\omega_0-1,\dots,0\}$. We then define $X_1=X_0+\omega_0+2\geq 1$,
\[\eta_{X_0}=\eta_{X_1}=0\]
and
\[\eta_{x}=1\qquad \mbox{ for } \qquad X_0< x<X_1.\]
Then, both $(\eta_{X_0-x})_{x\geq 0}$ and $(\eta_{X_1+x})_{x\geq 0}$ are Markov chains on $\{0,1\}$ with the same distribution, and with transition probabilities for $x\geq 0$
\begin{equation}
\label{eq:transitionrates1}
\P(\widetilde{\eta}_{x+1}=1\mid \widetilde{\eta}_x=1)=a(\rho), \qquad \P(\widetilde{\eta}_{x+1}=0\mid \widetilde{\eta}_x=1)=1-a(\rho),
\end{equation}
where $a(\rho)$ is the active density defined in \eqref{eq:arho} whereas 
\begin{equation}
\label{eq:transitionrates2}
\P(\widetilde{\eta}_{x+1}=1\mid \widetilde{\eta}_x=0)=1.
\end{equation}

\bigskip

We now prove \eqref{eq:mappingdistributions}. Denote by $\widetilde{\pi}_\rho$ the distribution following the construction above, we want to prove that $\widetilde{\pi}_\rho=\pi_\rho$, where $\pi_\rho$ is given by \eqref{pirho2}. We first show that $\widetilde{\pi}_\rho$ is translation invariant. To do so, fix a local configuration $\sigma$ on $\nota{\Lambda}=\{-\ell, \dots, \ell\}$, we can safely assume that there is at least an empty site to the left of the origin, and two empty sites to the right of the origin.  Otherwise, we can simply derive the value of $\widetilde{\pi}_\rho(\eta_{\vert  \Lambda}=\sigma)$ by  extending $\sigma$ on a larger set $\Lambda$. As a consequence of this assumption, the origin and site $1$'s clusters are both fully contained in $ \Lambda,$ where we call a \emph{particle cluster} an empty site followed by all its consecutive particles.  Two cases can arise: either $\sigma_1=1$, in which case the origin and site $1$ are in the same cluster. Then, since the position of the origin in the cluster is chosen uniformly, we have by construction
\[\widetilde{\pi}_\rho(\eta_{\vert  \Lambda}=\sigma)=\widetilde{\pi}_\rho(\eta_{\vert \tau_{-1} \Lambda}=\sigma_{\cdot+1}),\]
where we defined $\tau_{-1}\Lambda=\{-1-\ell, \ell-1\}$. The other possibility is that $\sigma_1=0$, in which case by translating the event $\{\eta_{\vert  \Lambda}=\sigma\}$  by $-1$, we change the origin's cluster. For this reason, we need to compute explicitly the probability of obtaining $\sigma$. In $\sigma$, denote by $X_0:=-\omega^\sigma_0-2$ the position of the first empty site left of the origin, $X_1=1$ by assumption is the position of the first empty site right of the origin, and $X_2=\omega^\sigma_1+3$ the position of the second empty site to the right of the origin. By construction, to obtain $\sigma$ under $\widetilde{\mu}_\rho$, we need 
\begin{enumerate}[1)]
\item To choose the right size for the central cluster (the one of the origin), which occurs with probability 
\[p_1:=(\omega^\sigma_0+2) a(\rho)^{\omega^\sigma_0}\rho(1-a(\rho))\] 
given by \eqref{eq:centralcluster}. 
\item To choose the right position for the central cluster, which occurs with probability $p_2:= 1/(\omega^\sigma_0+2)$.
\item To build  the next cluster right of the origin according to the Markovian construction, which then occurs with geometric probability 
\[p_3:=\mu_\rho(\omega_1=\omega^\sigma_1)=a(\rho)^{\omega^\sigma_1}(1-a(\rho))\] 
given by \eqref{eq:Defmurho}. 
\item To build, by the Markovian construction, the rest of the configuration $\sigma$ outside of those two clusters, which occurs with probability denote $p_4:=p_4(\sigma)$.
\end{enumerate}
With these notations, we can now write
\begin{equation}
\label{eq:murhosigma}
\widetilde{\pi}_\rho(\eta_{\vert  \Lambda}=\sigma)=p_1 p_2 p_3p_4.
\end{equation}

Similarly,  to obtain $\sigma_{\cdot +1}$ in the translated box $\tau_{-1} \Lambda$, we need 
\begin{enumerate}[1)]
\item To choose the right size for the central cluster, which has now changed because of the translation, this occurs with probability 
\[q_1:=(\omega^\sigma_1+2) a(\rho)^{\omega^\sigma_1}\rho(1-a(\rho))\] 
\item To choose the right position for the central cluster, which occurs with probability $q_2:= 1/(\omega^\sigma_1+2)$.
\item To build the cluster left of the origin according to the Markovian construction, which now occurs with probability 
\[q_3:=\mu_\rho(\omega_{-1}=\omega^\sigma_0)=a(\rho)^{\omega^\sigma_0}(1-a(\rho)).\] 
given by \eqref{eq:Defmurho}. 
\item To build by the Markovian construction the rest of the configuration $\sigma$ outside of those two clusters, which occurs with probability denote $q_4:=p_4(\sigma)$
\end{enumerate}
We obtain 
\begin{equation}
\label{eq:murhosigma2}
\widetilde{\pi}_\rho(\eta_{\vert \tau_{-1} \Lambda}=\sigma_{\cdot+1})=q_1 q_2 q_3q_4=\widetilde{\pi}_\rho(\eta_{\vert  \Lambda}=\sigma),
\end{equation}
which proves that $\widetilde{\pi}_\rho$ is translation invariant since $\Lambda$ is of arbitrary size.

\bigskip

By translation invariance, the average density $\E_{\widetilde{\pi}_\rho}(\eta_0)$ is the average density of an arbitrary cluster, which is 
\begin{equation}
\label{eq:Epitilde}
\E_{\widetilde{\pi}_\rho}(\eta_0)=\frac{1+\frac{a(\rho)}{1-a(\rho)}}{2+\frac{a(\rho)}{1-a(\rho)}}=\frac{1}{2-a(\rho)}=\rho.
\end{equation}
We now prove that $\widetilde{\pi}_\rho$ also satisfies the explicit formula \eqref{pirho2} for $\pi_\rho$. To do so, we are going to use the translation invariance of $\widetilde{\pi}_\rho$, and the fact that very far from the origin, the central cluster has no influence, and the configuration's construction is Markovian. For this purpose, choose $x\in \Z$, an ergodic  configuration $\sigma\in \mathcal{E}_{\Lambda}$ on $\nota{\Lambda}:=\Lambda_\ell(x)=\{x+1,\dots,x+\ell\}$, and denote by 
\begin{align*}
n_{\circ\bullet}&:=card\{z\in \Lambda \setminus\{x+\ell\}, \; \sigma_z=0, \;\sigma_{z+1}=1\},\\
n_{\bullet \circ}&:=card\{z\in \Lambda \setminus\{x+\ell\}, \; \sigma_z=1, \; \sigma_{z+1}=0\},\\
n_{\bullet\bullet}&:=card\{z\in \Lambda \setminus\{x+\ell\}, \;  \sigma_z=\sigma_{z+1}=1\}.\\
\end{align*}
Since $\sigma$ was assumed ergodic, $ n_{\circ\bullet}+n_{\bullet \circ}+n_{\bullet \bullet}=\ell-1$. Furthermore, denoting by $p=\sum_{z\in \Lambda}\sigma_z$ the number of particles in $\sigma$, $\ell-p$ is its number of empty sites, therefore
\[n_{\circ \bullet}=\ell-p-(1-\sigma_{x+\ell}) \qquad  \mbox{ and }\qquad n_{\bullet\circ}=\ell-p-(1-\sigma_{x+1}).\]
In other words, the number of $\circ \bullet$ is $\ell-p$ except if the last site is empty, in which case it is $\ell-p-1$, and similarly for $\bullet\circ$. From those identities, we finally obtain
\[n_{\bullet \bullet}=2p-\ell +1-\sigma_{x+1}-\sigma_{x+\ell}.\]

Fix now $x\gg1$, unless the central cluster reaches $x$, which occurs with probability exponentially small in $x$, we can use the markovian construction of $\widetilde{\pi}_\rho$. Since $\widetilde{\pi}_\rho(\eta_{x+1}=\sigma_{x+1})=\rho^{\sigma_{x+1}}(1-\rho)^{1-\sigma_{x+1}}$ according to \eqref{eq:Epitilde} and translation invariance, we obtain using the transition rates \eqref{eq:transitionrates1}, \eqref{eq:transitionrates2}
\begin{align*}
\widetilde{\pi}_\rho(\eta_{\mid \Lambda}=\sigma)&=\widetilde{\pi}_\rho(\eta_{\mid \Lambda\setminus\{x+1\}}=\sigma_{\mid \Lambda\setminus\{x+1\}})\widetilde{\pi}_\rho(\eta_{x+1}=\sigma_{x+1})+O(e^{-cx})\\
&=1^{n_{\circ \bullet}}(1-a(\rho))^{n_{\bullet \circ}}a(\rho)^{n_{\bullet \bullet}}\rho^{\sigma_{x+1}}(1-\rho)^{1-\sigma_{x+1}}+O(e^{-cx}).
\end{align*}
Letting $x\to\infty$ and using $\widetilde{\pi}_\rho$'s translation invariance, we recover \eqref{pirho2} and obtain as wanted that $\widetilde{\pi}_\rho=\pi_\rho$.

Thanks to   \eqref{eq:mappingdistributions}, we also have the following result.
\begin{corollary}
\label{cor:Markov}
Under $\pi_\rho$, both $(\eta_x)_{x\geq 0}$ and $(\eta_{-x})_{x\geq 0}$ are distributed as homogeneous Markov chains, started from $\eta_0\sim Bernoulli(\rho)$, and with transition probabilities given by \eqref{eq:transitionrates1} and \eqref{eq:transitionrates2}. However, these two Markov chains are \emph{not} independant.
\end{corollary}
\begin{proof}
For any $x_0\in \Z$, by construction, $\bb{P}(\eta_{x_0}=1)=\rho$, and the distribution of $(\eta_{x_0+x})_{x\geq 0}$ converges as $x_0\to\infty$ to that of  a homogeneous Markov chain with transition probabilities given by \eqref{eq:transitionrates1} and \eqref{eq:transitionrates2}, because the influence of the initial cluster's construction vanishes. But by translation invariance, this distribution does not depend on $x_0$, which proves the corollary.
\end{proof}

\subsection{Compressibility}
\label{sec:comp}
We prove here for the sake of completeness identity \eqref{eq:Defchirho} for the equilibrium compressibility for the FEP, namely
\begin{equation*}
\chi(\rho)=\sum_{x\in \Z}Cov_{\pi_\rho}(\eta_0,\eta_x)=\rho(1-\rho)(2\rho-1).
\end{equation*}
To prove it, it is convenient to consider the construction $\widetilde{\pi}_\rho$ for $\pi_\rho$ obtained in Section \ref{sec:statPi} rather than \eqref{pirho2}. We start by writing by translation invariance of $\pi_\rho$ that
\begin{equation}
\label{eq:chirho2}
\chi(\rho)=\lim_{y\to\infty} \sum_{x\in \Z}Cov_{\pi_\rho}(\eta_y,\eta_{x+y})=\rho(1-\rho)+ 2\lim_{y\to\infty} \sum_{x\geq 1}Cov_{\pi_\rho}(\eta_y,\eta_{x+y}).
\end{equation}
Because we send $y$ to $\infty$, we do not need to take into account the central cluster, so that the distribution of $(\eta_{x+y})_{x\geq 0}$, in the limit $y\to\infty$ is that of a Markov chain with transition probabilities given by \eqref{eq:transitionrates1} and \eqref{eq:transitionrates2}. Consider therefore the distribution $\nu_\rho$  of a markov chain $(\eta_{x})_{x\geq 0}$ with transition probabilities given by \eqref{eq:transitionrates1} and \eqref{eq:transitionrates2}.

We now define 
\[g_x=\nu_\rho(\eta_0\eta_x=1),\qquad \mbox{ and }\qquad h_x=g_x-\rho^2=Cov_{\nu_\rho}(\eta_0,\eta_x),\]
and straightforward computation yield $h_1=\rho(a(\rho)-\rho)=-(1-\rho)^2$. We then write 
\[g_{x+1}=\nu_\rho(\eta_0\eta_x\eta_{x+1}=1)+\nu_\rho(\eta_0 (1-\eta_x)\eta_{x+1}=1).\]
Because $\nu_\rho$ only charges the ergodic component, the second term in the right-hand side is equal to $\nu_\rho(\eta_0 (1-\eta_x)=1)=\rho-g_x$, whereas the first is $\nu_\rho(\eta_0\eta_x\eta_{x+1}=1)=a(\rho)g_x.$ This  straightforwardly yields 
\[h_{x+1}=-\frac{1-\rho}{\rho}h_x \qquad \Longrightarrow \qquad h_x=\rho(1-\rho) \pa{\frac{\rho-1}{\rho}}^x,\]
so that using \eqref{eq:chirho2}, we obtain
\begin{equation*}
\chi(\rho)=\rho(1-\rho)+2\sum_{x\geq 1}h_x=\rho(1-\rho)(2\rho-1)
\end{equation*}
as wanted.

\subsection{Stationary states and mapping}
\label{sec:SSmapping}
In this section, we comment on why the mapped distribution $\widehat{\mu}_\rho(\omega=\cdot)$ defined in Section \ref{sec:SSZR} is \emph{not} a stationary state for the zero-range generator \eqref{eq:Lzr}. The reason is straightforward, because we need to consider instead the generator $\genzrh$ on the mapped pair $(\omega, X_0)$. More precisely, given a FEP trajectory $\eta=\{\eta(t),\;t\geq 0\}$ the process $(\omega, X_0)(t):=\Pi(\eta(t))$ defined by the mapping $\Pi$ defined in Section \ref{sec:staticmapping} is a Markov process with generator  $\genzrh $ defined on functions $f$ on $\widehat{\mathcal{E}}_\Z$ (cf. \ref{eq:EZhat}) as
\begin{align*}
\genzrh f (\omega,x) &= \genzrstar f (\cdot,x)(\omega)\\
&+\mathbf{1}_{\{\omega_{-1} \geq 1\}} p_N  \Big\{f(\omega^{-1,0},x-1)  - f(\omega,x) \Big\} +\mathbf{1}_{\{\omega_1 \geq 1\}} q_N  \Big\{f(\omega^{1,0},x)- f(\omega,x) \Big\}\\
&+\mathbf{1}_{\{\omega_0 \geq 1\}} p_N  \Big\{f(\tau_{-1}(\omega^{0,1}),0)\mathbf{1}_{\{ x=-\omega_0-1\}}  +f(\omega^{0,1},x)\mathbf{1}_{\{ x>-\omega_0-1\}}- f(\omega,x) \Big\}\\
&+\mathbf{1}_{\{\omega_0 \geq 1\}} q_N  \Big\{f(\tau_{1}(\omega^{0,-1}),-\omega_{-1}-2)\mathbf{1}_{\{ x=0\}}  +f(\omega^{0,-1},x+1)\mathbf{1}_{\{ x<0\}}- f(\omega,x) \Big\}
\end{align*}
where $\genzrstar$ is the zero-range generator defined in \eqref{eq:Lzr}, except that all jumps to and from the origin are suppressed, and $(\tau_x\omega)_{y}=\omega_{y-x}$ is the  configuration translated by $x$. Note that the contribution of $\genzrstar$ does not affect the position of $x$, since the tagged empty site in the exclusion configuration only changes when an $\omega$-particle jumps over the edge $(-1,0)$, or when an $\omega$-particle jumps from site $0$ to site $1$ and an empty site is at the origin in $\eta$.

By applying the mapping back and forth, it is then straightforward to show that, for any function $f(\omega)$ of the zero-range configuration, 
\[\E_{\widehat{\mu}_\rho}(\genzrh f)=0,\]
so that the distribution $\widehat{\mu}_\rho$ is invariant w.r.t. the generator $\genzrh $ on  $\widehat{\mathcal{E}}_\Z$.
However, projected on $\omega$, $\genzrh$ is not the generator of the constant rate zero-range process, but rather the modified generator
\begin{multline*}
\genzrh f (\omega) = \genzr f (\omega)\\
+\mathbf{1}_{\{\omega_0 \geq 1,\; x=-\omega_0-1\}} p_N  \Big\{f(\tau_{-1}(\omega^{0,1}))\mathbf{1}_{\{ x=-\omega_0-1\}}  -f(\omega^{0,1}) \Big\}+\mathbf{1}_{\{\omega_0 \geq 1, \; x=0\}} q_N  \Big\{f(\tau_{1}(\omega^{0,-1}))-f(\omega^{0,-1})\Big\},
\end{multline*}
where the two last terms account for the configuration translation when the origin's cluster changes because an empty site jumped over the exclusion edge $(0,1)$. Note that obviously, the generator above is not a Markov generator on the set of zero-range configurations $\N^\Z$, since it is defined in parts by the external variable $x$. This identity shows why the mapped stationary distribution $\widehat{\mu}_\rho$ is \emph{not} invariant w.r.t. the generator of the constant rate zero-range process.

\bibliographystyle{plain}
\bibliography{reference.bib}

\end{document}